\documentclass[final,3p,times]{elsarticle}
\usepackage{lineno,hyperref}
\usepackage{amssymb,amsmath,amsthm}
\usepackage{hyperref}
\usepackage{enumitem}

\usepackage[OT1]{fontenc} 

\newtheorem{theorem}{Theorem}[section] 

\newtheorem{definition}{Definition}[section]
\newtheorem{proposition}[theorem]{Proposition}
\newtheorem{corollary}[theorem]{Corollary}
\newtheorem{remark}[theorem]{Remark}

\journal{}

\begin{document}

\begin{frontmatter}
\title{A fractional Laplacian problem with mixed singular nonlinearities and nonregular data}
\author[mymainaddress]{Masoud Bayrami-Aminlouee\corref{mycorrespondingauthor}}
\cortext[mycorrespondingauthor]{Corresponding author}
\ead{masoud.bayrami1990@student.sharif.edu}
\author[mymainaddress]{Mahmoud Hesaaraki}
\ead{hesaraki@sharif.edu}
\address[mymainaddress]{Department of Mathematical Sciences, Sharif University of Technology, Azadi Ave., P.O. Box 11365-9415, Tehran, Iran}

\begin{abstract}
In this note, we study on the existence and uniqueness of a positive solution to the following doubly singular fractional problem:
$$
\begin{cases}
(-\Delta )^s u = \dfrac{K(x)}{u^q} + \dfrac{f(x)}{u^{\gamma}}+\mu & \mathrm{in} \,\, \Omega,\\ u>0 & \mathrm{in} \,\, \Omega, \\ u=0 & \mathrm{in} \,\, \big(\mathbb{R}^N \setminus \Omega \big).
\end{cases}
$$
Here $\Omega \subset \mathbb{R}^N$ ($N > 2s$) is an open bounded domain with smooth boundary, $ s \in (0,1)$, $q>0$, $\gamma>0$, and $K(x)$ is a positive H\"older continuous function in which behaves as $\mathrm{dist}(x, \partial \Omega)^{-\beta}$ near the boundary with $0 \leq \beta <2s$. Also, $0 \leq f, \mu \in L^1(\Omega)$, or non-negative bounded Radon measures in $\Omega$. Moreover, we assume that $0<\frac{\beta}{s}+q < 1$, or $\frac{\beta}{s}+q > 1$ with $2\beta+q(2s-1)<(2s+1)$. For $s \in (0,\frac{1}{2})$, we take advantage of the convexity of $\Omega$. For any $\gamma>0$, we will prove the existence of a positive weak (distributional) solution to the above problem. Besides, for the case $ 0 < \gamma \leq 1$, $\mu \in L^1(\Omega)+X^{-s}(\Omega)$, and some weighted integrable functions $f$, we will show the existence and uniqueness of another notion of a solution, so-called entropy solution. Also, we will discuss the uniqueness of the weak solution for the case $\gamma>1$, and also the equivalence of entropy and weak solutions for the case $0< \gamma \leq 1$. Finally, we will have some relaxation on the assumptions of $f$ in order to prove the existence of solutions.
\end{abstract}


\end{frontmatter}

\section{Introduction}
This note is concerned with the existence of a unique positive solution to the following doubly singular elliptic problem: 
\begin{equation}
\label{Eq1}
\begin{cases}
(-\Delta )^s u = \dfrac{K(x)}{u^q}+\dfrac{f(x)}{u^{\gamma}}+\mu & \mathrm{in} \,\, \Omega,\\ u>0 & \mathrm{in} \,\, \Omega, \\ u=0 & \mathrm{in} \,\,  \big(\mathbb{R}^N \setminus \Omega \big).
\end{cases}
\end{equation}
Here $\Omega \subset \mathbb{R}^N$ ($N > 2s$) is an open bounded domain with smooth boundary, $ s \in (0,1)$, $q>0$, and $\gamma>0$. Also, we assume that $K \in C_{\mathrm{loc}}^{\theta}(\Omega)$, $\theta \in (0,1)$, satisfies the following condition:
\begin{equation}
\label{Con1}
C_1 \delta^{-\beta}(x) \leq K(x) \leq C_2 \delta^{-\beta}(x), \qquad \forall x \in \Omega,
\end{equation}
for some $0 \leq \beta<2s$ and $C_1,C_2>0$. Here $\delta(x):=\mathrm{dist}(x, \partial \Omega)$, $x \in \Omega$, is the distance function from the boundary $\partial \Omega$. Also, $0 \leq f, \mu \in L^1(\Omega)$, or non-negative bounded Radon measures in $\Omega$. Moreover, we assume that $0<\frac{\beta}{s}+q < 1$, or $\frac{\beta}{s}+q > 1$ with $2\beta+q(2s-1)<(2s+1)$. For $s \in (0,\frac{1}{2})$, we take advantage of the convexity of $\Omega$. The operator $(-\Delta)^{s} $ stands for the fractional Laplacian which is given by a singular integral operator in the following way:
$$
(-\Delta )^{s} u(x) = C_{N,s} \, \mathrm{P.V.} \int_{\mathbb{R}^N} \frac{u(x)-u(y)}{|x-y|^{N+2s}} \, dy
$$
where $\mathrm{P.V.}$ denotes the Cauchy principal value and $C_{N,s}= \frac{4^s \Gamma(\frac{N}{2}+s)}{\pi^{\frac{N}{2}} |\Gamma(-s)|}$, is the normalization constant such that the identity, $(-\Delta )^s u = \mathcal{F}^{-1} \big(|\xi|^{2s} \hat{u}(\xi) \big)$ holds. Here $\Gamma$ is the Gamma function, and $\mathcal{F}u =\hat{u}$ denotes the Fourier transform of $u$. For more details about fractional Laplacian and also for the basic properties of the fractional Laplace operator, see the papers \cite{MR2944369, MR2270163, daoud2021fractional} and also Chapter 5 of the book \cite{MR0290095}.

For any $\gamma>0$, we will prove the existence of a positive weak solution to problem \eqref{Eq1} which involves two singular nonlinearities. See \cite[section 6]{MR2592976} for a discussion about mixed nonlinearities. Since problem \eqref{Eq1} involves a nonregular term, $\mu$, concerning the uniqueness, it is natural to consider the notion of entropy solution, which will be defined later in section \ref{Section2}. The motivation of the definition comes from the works \cite{MR1760541, MR1354907, MR1409661, C2} (see \cite[section 1]{MR3456751} for a very quick introduction, concerning the uniqueness, for the notion of renormalized, entropy, and SOLA solutions). In fact, in the case $ 0 < \gamma \leq 1$, and $\mu \in L^1(\Omega)+X^{-s}(\Omega)$, we will show the existence and uniqueness of the entropy solution to problem \eqref{Eq1} for some weighted integrable functions $f$. It is important to note that in \cite[sections 3 and 4]{MR3935063}, the authors discussed uniqueness for the notion of very weak solution to problem \eqref{Eq1} by invoking a Kato type inequality in the case $0<\gamma<1$, with $K \equiv 0$. Also, see the article \cite{MR3639996}, where the authors investigated the existence and uniqueness of a notion of a solution for the fractional $p$-Laplacian case of \eqref{Eq1} with $K, \mu \equiv 0$. Besides, see the works \cite{MR2592976, MR3450747} for some relaxation assumptions on $f$ to prove the existence of solutions.

We also refer the readers to the work \cite{MR3456751} where the author studied some integro-differential equations involving measure data by the duality method. Also, see \cite{MR2779579} for the duality approach to the fractional Laplacian with measure data. ‌Besides, see the work \cite{MR3339179} in which authors developed an existence, regularity, and potential theory for nonlinear non-local equations involving $L^1$ and measure data. Furthermore, recently G\'omez-Castro, and V\'azquez, \cite{MR4026184}, found the necessary and sufficient condition for the existence of the solutions for a fractional Schr\"odinger equation with singular potential and measure data. Also, see \cite{MR3865202} for their similar work in the context of the weighted space approach. 

In \cite{MR1037213}, Lazer and McKenna proved that if $\Omega$ has a regular boundary, and $p(x)$ is regular on $\overline{\Omega}$, and $\alpha$ is any positive number; then there is a unique classical solution $u \in C^2(\Omega) \cap C(\overline{\Omega})$ to the following problem:
\begin{equation}
\label{MOTO}
\begin{cases}
-\Delta  u =\dfrac{p(x)}{u^{\alpha}}   & \mathrm{in} \,\, \Omega \\
u>0   & \mathrm{in} \,\, \Omega \\
u=0   & \mathrm{on} \,\, \partial \Omega,
\end{cases}
\end{equation}
where $\Omega$ is a bounded domain of $\mathbb{R}^N$. This solution does not belong to $C^2(\overline{\Omega})$. Moreover, they found a necessary and sufficient condition with respect to $\alpha$, that this unique solution has a finite Dirichlet integral. More precisely, they proved that $u \in H_0^1(\Omega)$, if and only if $\alpha<3$. Also, they showed that $u \not \in C^1(\overline{\Omega})$, if $\alpha>1$. In article \cite{MR3480326} Bougherara, Giacomoni and Hernandez proved the existence of solutions to problem \eqref{MOTO}, for any $\alpha \in \mathbb{R}$, and $p \in C(\Omega)$ behaves as $\delta(x)^{-\beta}$, near the boundary, with $0 \leq \beta <2$. They discussed the existence, uniqueness, and stability of the weak solution. Also, they proved accurate estimates on the gradient of the solution near the boundary. Then, they proved that the solution belongs to $W^{1,q}_0(\Omega)$ for $1<q<\frac{1+\alpha}{\alpha+\beta-1}$. Our results in this paper are in the similar direction for the non-local case. Problems as in \eqref{MOTO} have been extensively studied both for their pure mathematical interest, \cite{MR3893792, MR3489386, MR3712944,MR2592976,MR3450747,MR427826,MR3935063, MR3639996, MR3797738, MR3356049, MR3645936, MR4044739, MR3943307, MR3323892, MR3906415, MR2989693, MR4160003}, and for their relations with some physical phenomena in the theory of pseudoplastic fluids, \cite{MR564014}. Moreover, see \cite{C1} for a $p$-Laplacian evolution problem with singular terms and nonregular data. Besides, see \cite{latorre2020dirichlet}, where the authors studied the $1$-Laplacian case of Lazer-McKenna problem in a suitable functional setting. Also, see \cite{alves2020uniqueness} for the $p(x)$-Laplacian case.

It is important to point out that in \cite{MR3819101}, Chipot and De Cave introduced new techniques for solving some class of singular elliptic equations with source term having a singularity at the origin. More precisely, they analyzed the following problem:
\begin{equation}
\label{ChipotDeCave}
\begin{cases}
-\mathrm{div} \big(\mathcal{A}(x,\nabla u) \big) =H(u) \mu & \mathrm{in} \,\, \Omega \\
u>0 & \mathrm{in} \,\, \Omega \\
u=0 & \mathrm{on} \,\, \partial \Omega,
\end{cases}
\end{equation}
where $\Omega$ is a bounded or unbounded domain of $\mathbb{R}^N$, $H$ is (possibly) singular at the origin like $H(s)=s^{-\gamma}$, $\gamma>0$, $\mathcal{A}: \mathbb{R}^N \times \mathbb{R}^N \to \mathbb{R}^N$ satisfies Leray-Lions conditions of $p$-Laplace type with $1<p<N$, and $\mu$ is a suitable non-negative datum. The main idea of their method is based on using the Lax-Milgram Theorem. Their technique's novelty is that it applies to the unbounded case of $\Omega$, too. They avoid using the notion of renormalized solution (introduced in \cite{MR1760541} for measure data problems) and instead use a suitable class of test functions in the weak (variational) formulation of problem \eqref{ChipotDeCave}. Also, see \cite{MR4027821, MR4102796} for the similar papers.

Adimurthi, Giacomoni and Santra \cite{MR3797614} investigated the existence and bifurcation results to the following equation involving the fractional Laplacian with singular nonlinearity:
\begin{equation}
\label{Eq3.01}
\begin{cases}
(-\Delta )^s u = \lambda \big(K(x)u^{-\delta} + f(u) \big) &   \mathrm{in} \,\, \Omega, \\
u>0 &   \mathrm{in} \,\, \Omega, \\ 
u=0 &   \mathrm{in} \,\, \big(\mathbb{R}^N \setminus \Omega \big). 
\end{cases}
\end{equation}
Here $\lambda>0$ and $f : \mathbb{R}^+ \to \mathbb{R}^+$ is a positive $C^2$ function and $K(x)$ is a H\"older continuous function in which behaves as $\mathrm{dist}(x,\partial \Omega)^{-\beta}$ near the boundary with $0 \leq \beta <2s$. First, for any $\delta>0$ and for $\lambda >$ small enough, they proved the existence of solutions to \eqref{Eq3.01}. Next, for a suitable range of values of $\delta$, they showed the existence of an unbounded connected branch of solutions to \eqref{Eq3.01} emanating from the trivial solution at $\lambda=0$. Also, for a certain class of nonlinearities $f$, they derived a global multiplicity result. 

In \cite{CRMATH_2020__358_2_237_0} Arora, Giacomoni, Goel, and Sreenadh studied the symmetry and monotony of the positive solutions to problem \eqref{Eq3.01} with $\lambda=1$, $K \equiv 1$, and $f(u)$ a locally Lipschitz function. For that, they implemented the moving plane method. Then by using this symmetry result, they investigated the global behavior of solutions.

In section \ref{Section2}, we will introduce the functional setting. Also, after defining the notions of the weak solution and entropy solution to problem \eqref{Eq1}, we will outline our theorems about the existence and uniqueness results to problem \eqref{Eq1}. In section \ref{Section3}, we will provide proof of these results, which extends our previous study \cite{bayrami2020existence}. In section \ref{Section4}, we will discuss the relaxing some assumptions on $f$ in order to still have the existence results.

\section{Functional setting and main result}
\label{Section2}
Let $ 0 < s <1 $ and $ 1 \leq p < \infty $. The classical fractional Sobolev space defines as follows:
$$ W^{s,p}(\mathbb{R}^N) = \Bigg\{ u \in L^p(\mathbb{R}^N) \, : \, \int_{\mathbb{R}^N} \int_{\mathbb{R}^N} \frac{|u(x)-u(y)|^p}{|x-y|^{N+ps}} \, dxdy < \infty \Bigg\} $$
endowed with the Gagliardo norm:
$$ \|u\|_{W^{s,p}(\mathbb{R}^N)} = \|u\|_{L^p(\mathbb{R}^N)} + \Bigg( \int_{\mathbb{R}^N} \int_{\mathbb{R}^N} \frac{|u(x)-u(y)|^p}{|x-y|^{N+ps}} \, dxdy \Bigg)^{\frac{1}{p}}. $$
Also, we define 
$$ 
X^{s,p}(\Omega) =\Bigg\{ u: \mathbb{R}^N \to \mathbb{R} \,\, \mathrm{measurable}, \, u|_{\Omega} \in L^p(\Omega),\, \iint_{D_{\Omega}} \frac{|u(x)-u(y)|^p}{|x-y|^{N+ps}} \, dxdy < \infty 
\Bigg\}, 
$$
where $D_{\Omega}= \mathbb{R}^N \times \mathbb{R}^N \setminus \Omega^{c} \times \Omega^{c} $, with $\Omega^{c}=\mathbb{R}^N \setminus \Omega$ and $\Omega$ is a bounded smooth domain in $\mathbb{R}^N$. This is a Banach space with the following norm:
\begin{equation}
\label{Eq3.1}
\|u\|_{X^{s,p}(\Omega)} = \Bigg( \int_{\Omega} |u|^p \, dx + \iint_{D_{\Omega}} \frac{|u(x)-u(y)|^p}{|x-y|^{N+ps}} \, dxdy \Bigg)^{\frac{1}{p}}. 
\end{equation}
In the case $p=2$, we denote by $X^s(\Omega)$ the space $X^{s,2}(\Omega)$ which is a Hilbert space with the following scalar product:
$$ \langle u,v \rangle_{X^s(\Omega)} = \int_{\Omega} uv \, dx + \iint_{D_{\Omega}} \frac{(u(x)-u(y))(v(x)-v(y))}{|x-y|^{N+2s}} \, dxdy. $$
Moreover, we define $ X_0^{s,p}(\Omega) = \{ u \in X^{s,p}(\Omega) \, : \, u=0 \,\, \mathrm{a.e. \,\, in} \,\,  (\mathbb{R}^N \setminus \Omega ) \}$. This space, $X_0^{s,p}(\Omega)$, can also be identified by the closure of $C_c^{\infty}(\Omega)$ in $X^{s,p}(\Omega)$. Also, we let $X_0^s(\Omega)$ denotes $X_0^{s,2}(\Omega)$. It is easy to see that: 
$$ \Bigg( \int_{\mathbb{R}^N} \int_{\mathbb{R}^N} \frac{|u(x)-u(y)|^p}{|x-y|^{N+ps}} \, dxdy \Bigg)^{\frac{1}{p}} = \Bigg( \iint_{D_{\Omega}} \frac{|u(x)-u(y)|^p}{|x-y|^{N+ps}} \, dxdy \Bigg)^{\frac{1}{p}}, $$
for any $u \in X_0^{s,p}(\Omega)$. This equality defines another norm equivalent to the norm \eqref{Eq3.1} for $X_0^{s,p}(\Omega)$. We denote this norm by $\|u\|_{X_0^{s,p}(\Omega)}$, i.e.
$$ \|u\|_{X_0^{s,p}(\Omega)} = \Bigg( \iint_{D_{\Omega}} \frac{|u(x)-u(y)|^p}{|x-y|^{N+ps}} \, dxdy \Bigg)^{\frac{1}{p}}. $$
Then there exists a positive constant $C$ such that the following inequalities hold for all $u \in X_0^{s,p}(\Omega)$.
$$ \|u\|_{X_0^{s,p}(\Omega)} \leq \|u\|_{W^{s,p}(\mathbb{R}^N)} \leq C \|u\|_{X_0^{s,p}(\Omega)}. $$
It is worth mentioning that the continuous embedding of $X_0^{s_2}(\Omega)$ into $X_0^{s_1}(\Omega)$, holds for any $s_1<s_2$. Besides, for the Hilbert space case, we have:
\begin{equation*}
\label{Eqre3.2}
\|u\|_{X_0^s(\Omega)}^2 = 2 C_{N,s}^{-1} \| (-\Delta)^{\frac{s}{2}} u\|_{L^2(\mathbb{R}^N)}^2,
\end{equation*}
where $C_{N,s}$ is the normalization constant in the definition of $(-\Delta)^s$. 

For the proofs of the above facts see \cite[subsection 2.2]{MR2879266} and \cite{MR2944369}. Also see \cite[subsection 3.2]{daoud2021fractional}.

For $ 0<r<\infty$, the Marcinkiewicz space $M^r(\Omega)$, is the set of all measurable functions $u: \Omega \to \mathbb{R}$, such that there exists $C>0$ with the following condition:
$$  \omega \Big( \Big\{x \in \Omega \, : \, |u(x)| \geq t \Big\} \Big) \leq \frac{C}{t^r}, \qquad \forall t>0. $$
Here $\omega$ denotes the Lebesgue measure on $\mathbb{R}^N$. This space is endowed with the following norm:
$$ \| u\|_{M^r(\Omega)} = \sup_{t>0} \, t \Bigg(\omega \Big( \Big\{x \in \Omega \, : \, |u(x)| \geq t \Big\} \Big) \Bigg)^{\frac{1}{r}}. $$
For every $1 < r < \infty$ and $0<\epsilon \leq r-1$, the following continuous embeddings hold, \cite{MR1354907}:
\begin{equation}
\label{Eq3.2}
L^r(\Omega) \hookrightarrow M^r(\Omega) \hookrightarrow L^{r-\epsilon}(\Omega).
\end{equation}
Also, the following continuous embedding will be used in this paper.
\begin{equation}
\label{Eq3.350}
X_0^{s,p}(\Omega) \hookrightarrow L^t(\Omega), \qquad \forall t \in[1, p_s^*],
\end{equation}
where $p_s^*=\frac{pN}{N-ps}$ is the Sobolev critical exponent. Moreover, this embedding is compact for $1 \leq t < p_s^*$. See \cite[Theorem 6.5 and Theorem 7.1]{MR2944369}.

Also we denote by $X_{\mathrm{loc}}^{s,p}(\Omega)$, the set of all functions $u$ such that $u \phi \in X_0^{s,p}(\Omega)$ for any $\phi \in C_c^{\infty}(\Omega)$. When we say $\{u_n\} \subset X_{\mathrm{loc}}^{s,p}(\Omega)$ is bounded, we mean that $\{\phi u_n\} \subset X_0^{s,p}(\Omega)$ is bounded for any fixed $\phi \in C_c^{\infty}(\Omega)$.

Since we are dealing with the non-local operator $(-\Delta)^s$, a new class of test functions should be defined precisely instead of the usual one $C_c^{\infty}(\Omega)$, i.e.
$$ \mathcal{T}(\Omega)= \Bigg \{ \phi:\mathbb{R}^N \to \mathbb{R}  \,\, \Bigg|\,\, \begin{aligned} & (-\Delta)^s \phi = \varphi, \,\,  \varphi \in  C^{\infty}_c(\Omega), \,\, \\
& \phi = 0 \,\,  \mathrm{in} \,\, \mathbb{R}^N \setminus \tilde{\Omega}, \,\,\mathrm{for \,\, some}  \,\, \tilde{\Omega} \Subset \Omega \end{aligned} \Bigg\}. $$
It can be shown that $\mathcal{T}(\Omega) \subset X_0^s(\Omega) \cap L^{\infty}(\Omega)$. Moreover, for every $\phi \in \mathcal{T}(\Omega)$, there exists a constant $\alpha \in (0,1)$ such that $\phi \in C^{0,\alpha}(\Omega)$. See \cite{MR3479207, MR3356049, MR3393266}.
It is easy to check that for $u\in X_0^s(\Omega)$ and $ \phi \in \mathcal{T}(\Omega)$:

\begin{align}
\nonumber
2 C_{N,s}^{-1} \int_{\mathbb{R}^N} u (-\Delta)^{s} \phi \, dx &= 2 C_{N,s}^{-1} \int_{\mathbb{R}^N}  (-\Delta)^{\frac{s}{2}} u (-\Delta)^{\frac{s}{2}} \phi \, dx \\
& = \iint_{D_{\Omega}} \frac{\big(u(x)-u(y) \big) \big(\phi(x)-\phi(y) \big)}{|x-y|^{N+2s}} \, dxdy. \label{Eq3.349n}
\end{align}

This equality shows the self-adjointness of $(-\Delta)^s$ in $X_0^s(\Omega)$. Also, one can show that $ (-\Delta)^s : X_0^s(\Omega) \to X^{-s}(\Omega) $ is a continuous strictly monotone operator, where $X^{-s}(\Omega)$ indicates the dual space of $X_0^s(\Omega)$.

\begin{definition}
We say that $u$ is a weak solution to \eqref{Eq1} if
\begin{itemize}
\item
$u \in L^1_{\mathrm{loc}}(\Omega)$, and for every $K \Subset \Omega$, there exists $C_K>0$ such that $u(x) \geq C_K$ a.e. in $K$ and also $u \equiv 0$ in $\big(\mathbb{R}^N \setminus \Omega \big)$;
\item
Equation \eqref{Eq1} is satisfied in the sense of distributions with the class of test functions $\mathcal{T}(\Omega)$, i.e.
\begin{equation}
\label{Eq3.4}
\int_{\mathbb{R}^N} u (-\Delta)^s \phi \, dx = \int_{\Omega} \frac{K(x)}{u^{q}} \phi \, dx + \int_{\Omega} \frac{f(x)}{u^{\gamma}} \phi \, dx  + \int_{\Omega} \mu \phi \, dx,
\end{equation}
for any $\phi \in \mathcal{T}(\Omega)$.
\end{itemize}
\end{definition}
\begin{remark}
\label{REMM01}
Note that since $\phi$ has compact support in $\Omega$, the first and second term on the right-hand side of \eqref{Eq3.4} are well defined by the strict positivity of $K(x)$ and $u$ on the compact subsets of $\Omega$. Also, we are allowed to assume that $\mu$ is a bounded Radon measure in $\Omega$ since the test functions belong to $C_c(\Omega)$. Moreover, notice that if the datum $\mu$ is regular enough to give the result that the solution $u$ satisfies $\phi u^{-\gamma} \in X_0^s(\Omega) \cap L^{\infty}(\Omega)$, for any $\phi \in \mathcal{T}(\Omega)$, then we may assume $f \in L^1(\Omega)+X^{-s}(\Omega)$, because of the duality between $L^1(\Omega)+X^{-s}(\Omega)$ and $X_0^{s}(\Omega) \cap L^{\infty}(\Omega)$. Since $\phi \in \mathcal{T}(\Omega)$ has compact support and $u$ is strictly positive on compact subsets of $\Omega$, this condition holds if $u^{-\gamma} \in X_{\mathrm{loc}}^s(\Omega)$. On the other hand, it is easy to see that since the map $ t \mapsto t^{-\gamma}$ is a Lipschitz function away from zero; thus $u^{-\gamma} \in X_{\mathrm{loc}}^s(\Omega)$ holds if $u \in X_{\mathrm{loc}}^s(\Omega)$.
\end{remark}

Concerning the uniqueness, we have another definition to solutions of \eqref{Eq1}. In fact we would like to consider the entropy solution. We will denote
$$
T_k(s) = \begin{cases}
s   & |s| \leq k \\
k \, \mathrm{sign} (s)   & |s| \geq k,
\end{cases}
$$
the usual truncation operator.
\begin{definition}
\label{DEFN}
Let $\mu \in L^1(\Omega)+X^{-s}(\Omega)$. We say that $u$ is an entropy solution to \eqref{Eq1} if
\begin{itemize} 
\item
for every $K \Subset \Omega$, there exists $C_K>0$ such that $u(x) \geq C_K$ in $K$ and also $u \equiv 0$ in $\big(\mathbb{R}^N \setminus \Omega \big)$;
\item
$T_k(u) \in X_0^s(\Omega)$, for every $k$, and $u$ satisfies the following family of inequalities:
\begin{equation}
\label{Eq3.500}
\int_{\{|u-\phi| < k \}} (-\Delta)^{\frac{s}{2}} u (-\Delta)^{\frac{s}{2}} (u-\phi) \, dx \leq \int_{\Omega} \frac{K(x)}{u^{q}} T_k(u-\phi) \, dx + \int_{\Omega} \frac{f(x)}{u^{\gamma}} T_k(u-\phi) + \int_{\Omega}  T_k(u-\phi) \, d\mu
\end{equation}
for any $k$ and any $\phi \in X_0^s(\Omega) \cap L^{\infty}(\Omega)$. 
\end{itemize}
We will see later that the first and second terms on the right-hand side are well defined after the construction of the entropy solution. Moreover, the last term $\int_{\Omega} T_k(u-\phi) \, d\mu $, is well-posed because of the duality between $L^1(\Omega)+X^{-s}(\Omega)$ and $X_0^s(\Omega)\cap L^{\infty}(\Omega)$. Notice that $L^1(\Omega)+X^{-s}(\Omega)$ embeds in the dual space of $X_0^s(\Omega)\cap L^{\infty}(\Omega)$.
\end{definition} 

\begin{theorem}
\label{Thm1}
Let $ s \in (0,1)$, $q>0$, $\gamma>0$, and $ 0 \leq f, \mu \in L^1(\Omega)$. Also assume that $K \in C^{\theta}_{\mathrm{loc}}(\Omega)$, $\theta \in (0,1)$, satisfies assumption \eqref{Con1}, and $0<\frac{\beta}{s}+q < 1$, or $\frac{\beta}{s}+q > 1$ with $2\beta+q(2s-1)<(2s+1)$. Then there is a positive weak solution to problem \eqref{Eq1}. More precisely,
\begin{enumerate}[leftmargin=*, label=(\roman*)]
\item
If $0< \gamma \leq 1$, then $u \in X_0^{s_1,p}(\Omega)$, for all $s_1<s$ and for all $p<\frac{N}{N-s}$. Moreover, $T_k(u) \in X_0^s(\Omega)$, for any $k>0$. 
\item
If $ \gamma > 1$, then $u \in X_{\mathrm{loc}}^{s_1,p}(\Omega)$, for all $s_1<s$ and for all $p<\frac{N}{N-s}$. Moreover, $T_k^{\frac{\gamma+1}{2}}(u) \in X_0^s(\Omega)$, and $T_k(u) \in X_{\mathrm{loc}}^s(\Omega)$, for any $k>0$.  
\end{enumerate}
Moreover, in the case $0<\gamma \leq 1$ if:
\begin{equation}
\label{EnASSUMPTIONS}
\begin{cases}
\displaystyle\int_{\Omega} f \delta^{-s\gamma} \, dx < + \infty \,\, & 0<\dfrac{\beta}{s}+q<1, \\
\\
\displaystyle\int_{\Omega} f\delta^{-\frac{2s-\beta}{q+1} \gamma} \, dx < + \infty \,\, & \dfrac{\beta}{s}+q>1, \, \text{and} \,\,\,  2\beta+q(2s-1)<(2s+1), \\
\end{cases}
\end{equation}
then, there exists a unique positive entropy solution to problem \eqref{Eq1}.
\end{theorem}

\begin{theorem}
\label{Thm2}
Under the hypothesis of Theorem \ref{Thm1}, if we further assume that
\begin{equation}
\label{ASSUMPTIONS}
\begin{cases}
\displaystyle\int_{\Omega} f^2 \delta^{2s(1-\gamma)} \, dx < + \infty \,\, & 0<\dfrac{\beta}{s}+q<1, \\
\\
\displaystyle\int_{\Omega} f^2\delta^{2s-2\gamma \frac{2s-\beta}{q+1}} \, dx < + \infty \,\, & \dfrac{\beta}{s}+q>1, \, \text{and} \,\,\, 2\beta+q(2s-1)<(2s+1), \\
\end{cases}
\end{equation}
then the uniqueness is guaranteed. More precisely,
\begin{enumerate}[leftmargin=*, label=(\roman*)]
\item
If $ \gamma > 1$, then the weak solution obtained in Theorem \ref{Thm1} is unique.  
\item
If $0< \gamma \leq 1$, then the weak solution and the entropy solution obtained in Theorem \ref{Thm1} are equivalent.
\end{enumerate}
\end{theorem}

\section{Proof of Theorem \ref{Thm1} and Theorem \ref{Thm2}}
\label{Section3}

First of all, we consider the following auxiliary problem: 
\begin{equation}
\label{Eq4}
\begin{cases}
(-\Delta )^s u =  K(x)u^{-q} +  g &   \mathrm{in} \,\, \Omega, \\ u>0 &   \mathrm{in} \,\, \Omega, \\ u=0 &   \mathrm{in} \,\, \big(\mathbb{R}^N \setminus \Omega \big),
\end{cases}
\end{equation}
where $ 0 \leq g \in L^{\infty}(\Omega)$. The existence and uniqueness of a weak energy solution to problem \eqref{Eq4} is known, see \cite[Theorem 1.2 and Theorem 1.3 and also Remark 1.5]{MR3797614}. For the reader convenience, we recall it in the following Proposition. Before we get into the Proposition, we need to define the set $\mathcal{C}_q$ as the set of all functions in $u \in L^{\infty}(\Omega)$ such that there exist $k_1,k_2>0$ such that
\begin{equation}
\label{Eq4.5}
\begin{cases}
k_1 \delta^s(x) \leq u(x) \leq k_2 \delta^s(x), \,\, & 0<\frac{\beta}{s}+q<1, \\
k_1 \delta^{\frac{2s-\beta}{q+1}} (x) \leq u(x) \leq k_2 \delta^{\frac{2s-\beta}{q+1}} (x) , \,\, & \frac{\beta}{s}+q>1,
\end{cases}
\end{equation}

\begin{proposition}
\label{Pro1}
If $ g \in L^{\infty}(\Omega) $, $g \geq 0$, $s \in (0,1)$, $q>0$, and $K \in C^{\theta}_{\mathrm{loc}}(\Omega)$, $\theta \in (0,1)$, satisfies assumption \eqref{Con1}, then
\begin{enumerate}[leftmargin=*, label=(\roman*)]
\item
If $0<\frac{\beta}{s}+q < 1$, then there exists a unique weak energy solution to problem \eqref{Eq4} in $ X_0^s(\Omega) \cap \mathcal{C}_q \cap C^{0,s}(\mathbb{R}^N)$.
\item
If $\frac{\beta}{s}+q > 1$, and $2\beta+q(2s-1)<(2s+1)$, then there exists a unique weak energy solution to problem \eqref{Eq4} in $ X_0^s(\Omega) \cap \mathcal{C}_q \cap C^{0,\frac{2s-\beta}{(q+1)s}}(\mathbb{R}^N)$.
\end{enumerate}
\end{proposition}
The notion of the solution to \eqref{Eq4} is as follows. The function $u \in X_0^s(\Omega)$ is a weak energy solution to problem \eqref{Eq4} if
\begin{itemize} 
\item
for every $K \Subset \Omega$, there exists $C_K>0$ such that $u(x) \geq C_K$ in $K$ and also $u \equiv 0$ in $\big(\mathbb{R}^N \setminus \Omega \big)$;
\item
for every $\phi \in X_0^s(\Omega)$, we have:
\begin{equation}
\label{Eq4.600}
\int_{\mathbb{R}^N} (-\Delta)^{\frac{s}{2}} u (-\Delta)^{\frac{s}{2}} \phi \, dx = \int_{\Omega} K(x)u^{-q} \phi \, dx + \int_{\Omega} g \phi \, dx.
\end{equation}
\end{itemize}
Note that the first term on the right-hand side of the above equality is well defined by \eqref{Eq4.5} and applying the H\"older inequality and the fractional Hardy-Sobolev inequality (and convexity of $\Omega$ only for $0<s<\frac{1}{2}$), \cite[Theorem 1.1]{MR3814363}. More precisely, since $K(x)$ behaves like $\delta^{-\beta}(x)$, near the boundary and 
\begin{equation}
\label{EqBNTBOU}
u(x) \sim \begin{cases}
\delta^s (x), \qquad & 0<\frac{\beta}{s}+q<1, \\
\delta^{\frac{2s-\beta}{q+1}} (x), \qquad & \frac{\beta}{s}+q>1
\end{cases}
\end{equation}
near the boundary, we get the following estimates which will be used several times later. We distinguish some cases:
\begin{enumerate}[leftmargin=*, label=(\alph*)]
\item
$0<\frac{\beta}{s}+q < 1$.
\item
$\frac{\beta}{s}+q > 1$, with $2\beta+q(2s-1)<(2s+1)$.
\end{enumerate}
\subsection*{Case (a): Estimate for the case $0<\frac{\beta}{s}+q < 1$.}

\begin{align} 
\nonumber
\int_{\Omega} K(x) \frac{\phi}{u^q} \, dx \sim \int_{\Omega} \frac{\phi}{\delta^{sq+\beta}} \, dx  & \leq C_1 \Bigg(\int_{\Omega} \frac{\phi^2}{\delta^{2sq+2\beta}} \, dx\Bigg)^{\frac{1}{2}} \\
& \leq C_2 \|\phi\|_{X_0^{sq+\beta}(\Omega)} \leq C_3 \|\phi\|_{X_0^{s}(\Omega)}. \label{ESTIM1}
\end{align} 
Here in the last inequality we have used $0<\frac{\beta}{s}+q < 1$, and the continuous embedding of $X_0^{s_2}(\Omega)$ into $X_0^{s_1}(\Omega)$, for any $s_1<s_2$.
\subsection*{Case (b): Estimate for the case $ \frac{\beta}{s}+q > 1$.}
$$
\begin{aligned}
\int_{\Omega} K(x) \frac{\phi}{u^q} \, dx \sim \int_{\Omega} \frac{\phi}{\delta^{q{\frac{2s-\beta}{q+1}}+\beta}} \, dx & \leq \Bigg( \int_{\Omega} \frac{1}{\delta^{{(2s-\beta)\frac{q-1}{q+1}}+\beta}} \, dx \Bigg)^{\frac{1}{2}} \Bigg( \int_{\Omega} \frac{\phi^2}{\delta^{2s}} \, dx \Bigg)^{\frac{1}{2}}\\
& \leq C_4\Bigg( \int_{\Omega} \frac{1}{\delta^{{(2s-\beta)\frac{q-1}{q+1}}+\beta}} \, dx \Bigg)^{\frac{1}{2}} \|\phi\|_{X_0^{s}(\Omega)}.
\end{aligned}
$$
If in addition we assume $2\beta+q(2s-1)<(2s+1)$, i.e. $(2s-\beta)\frac{q-1}{q+1}+\beta<1$, then
\begin{equation}
\label{ESTIM2}
\int_{\Omega} K(x) \frac{\phi}{u^{q}} \, dx \leq C_5 \|\phi\|_{X_0^{s}(\Omega)}.  
\end{equation}

For general domains with some boundary regularity, the fractional Hardy-Sobolev inequality is proved for $s \in[ \frac{1}{2},1)$. See \cite{MR3933834, MR2085428, MR3021545}. But in \cite{MR3814363}, the authors proved the fractional Hardy-Sobolev inequality for any $s \in (0,1)$, by using the fact that the domain is a convex set and its distance from the boundary is a superharmonic function.

Now, for every $v \in L^2(\Omega)$, define $\Phi(v)=w$ where $w$ is the solution to the following problem for any fixed $n$:
\begin{equation}
\label{Eq5}
\begin{cases}
(-\Delta )^s w = K(x)w^{-q} +f_n(x)(|v|+\frac{1}{n})^{-\gamma}+\mu_n &   \mathrm{in} \,\, \Omega,\\ w>0 &   \mathrm{in} \,\, \Omega, \\ w=0 &   \mathrm{in} \,\,  \big(\mathbb{R}^N \setminus \Omega \big). 
\end{cases}
\end{equation}
Here $ f_n = T_n(f) $ and $ \mu_n = T_n(\mu) $ are the truncations at level $n$. If we show that $\Phi : L^2(\Omega) \to L^2(\Omega) $ has a fixed point $w_n$, then $w_n \in L^2(\Omega)$ will be the weak solution to the following problem in $X_0^s(\Omega) \cap \mathcal{C}_q$.
\begin{equation}
\label{Eq4.75}
\begin{cases}
(-\Delta )^s w_n =  K(x)w_n^{-q} + \dfrac{f_n(x)}{(w_n+\frac{1}{n})^{\gamma}} + \mu_n &   \mathrm{in} \,\, \Omega, \\ w_n>0 &   \mathrm{in} \,\, \Omega, \\ w_n=0 &   \mathrm{in} \,\, \big(\mathbb{R}^N \setminus \Omega \big).
\end{cases}
\end{equation}
For this purpose, we apply the Schauder's fixed-point theorem (see for example \cite[Theorem 3.2.20]{MR3890060}). We need to prove that $\Phi$ is continuous, compact and there exists a bounded convex subset of $L^2(\Omega)$ which is invariant under $\Phi$. 

For continuity let $v_k \to v$ in $L^2(\Omega)$. It is obvious that for each $n$:
$$ \Bigg\| \Bigg(f_n \Big(|v_k|+\frac{1}{n} \Big)^{-\gamma} \Bigg)-\Bigg(f_n \Big(|v|+\frac{1}{n} \Big)^{-\gamma}\Bigg) \Bigg\|_{L^2(\Omega)} \to 0, \qquad k \to \infty. $$
Now, from the uniqueness of the weak solution to \eqref{Eq4}, we conclude $\Phi(v_k) \to \Phi(v)$. 

For compactness, we argue as follows. For $v \in L^2(\Omega)$, let $w$ be the solution to \eqref{Eq5}. If $\lambda_1^s(\Omega)$ is the first eigenvalue of $(-\Delta)^s$ in $X_0^s(\Omega)$, \cite[Proposition 9]{MR3002745}, then we have
\begin{equation}
\label{Eq7}
\lambda_1^s(\Omega) \int_{\Omega} w^2 \, dx \leq \int_{\mathbb{R}^N} | (-\Delta)^{\frac{s}{2}} w|^2 \, dx.
\end{equation}
Testing \eqref{Eq5} with $\phi=w$, we have
\begin{equation}
\label{Eq8}
\int_{\mathbb{R}^N} | (-\Delta)^{\frac{s}{2}} w|^2 \, dx = \int_{\Omega} K(x)w^{1-q} \, dx + \int_{\Omega} \frac{f_n w}{(|v|+\frac{1}{n})^{\gamma}} \, dx + \int_{\Omega} w \mu_n \, dx.
\end{equation}
For the first term, invoking estimates \eqref{ESTIM1} and \eqref{ESTIM2}, we have
\begin{equation}
\label{Eq.8t}
\int_{\Omega} K(x)w^{1-q} \, dx = \int_{\Omega} \frac{K(x)w}{w^q} \, dx \leq C_1 \|w\|_{X_0^s(\Omega)}.
\end{equation}
For the second term on the right-hand side of \eqref{Eq8} we have the following estimate:
\begin{align}
\int_{\Omega} \dfrac{f_n}{(|v|+\frac{1}{n})^{\gamma}} w \, dx & \leq n^{\gamma+1} \int_{\Omega} w \, dx \leq C_2 \Big( \int_{\Omega} |w|^2 \, dx \Big)^{\frac{1}{2}},  \label{Eq9}
\end{align}
where in the last inequality we have used the H\"older inequality. Once more using the H\"older inequality gives $ \int_{\Omega} f_n w \, dx \leq C_3 \Big( \int_{\Omega} |w|^2 \, dx \Big)^{\frac{1}{2}}$ for some $C_3>0$. Thus combining this inequality with \eqref{Eq8}, \eqref{Eq.8t}, and \eqref{Eq9} we obtain
$$ 
\|w\|_{X_0^s(\Omega)}^2 \leq  C_1\|w\|_{X_0^s(\Omega)} + C_4 \Big( \int_{\Omega} |w|^2 \, dx \Big)^{\frac{1}{2}}.
$$
This inequality together with the embedding \eqref{Eq3.350} implies that $\|w\|_{X_0^s(\Omega)}=\|\Phi(v)\|_{X_0^s(\Omega)}$ is uniformly bounded in $v \in L^2(\Omega)$. Thus \eqref{Eq7} implies that $\Phi(L^2(\Omega))$ is contained in a ball of finite radius in $L^2(\Omega)$. Therefore this ball is invariant under $\Phi$. Again, by using the compactness of the embedding \eqref{Eq3.350} and because of the uniform boundedness of $\|\Phi(v)\|_{X_0^s(\Omega)}$, we get that $\Phi(L^2(\Omega))$ is relatively compact in $L^2(\Omega)$. This completes the proof of the compactness of $\Phi$.

\begin{proposition}
\label{Pro2}
Let $\{w_n\}$ be the solutions to \eqref{Eq4.75}. Then we have the following \textit{a priori} estimates.
\begin{enumerate}[leftmargin=*, label=(\roman*)]
\item
If $0< \gamma \leq 1$, then $\{T_k(w_n)\}_{n=1}^{\infty}$ is bounded in $X_0^s(\Omega)$.
\item
If $ \gamma > 1$, then $\{T_k^{\frac{\gamma+1}{2}}(w_n)\}_{n=1}^{\infty}$ is bounded in $X_{0}^s(\Omega)$ and $\{T_k(w_n)\}_{n=1}^{\infty}$ is bounded in $X_{\mathrm{loc}}^s(\Omega)$.
\end{enumerate}

\end{proposition} 
\begin{proof}
We will follow the proof of \cite[Theorem 5.2]{MR3479207}. 
\hfil \newline
\textbf{First case: boundedness of $\{T_k(w_n)\}_{n=1}^{\infty}$ in $X_0^s(\Omega)$ for $0<\gamma \leq 1$.}

Testing \eqref{Eq4.75} with $\phi=T_k(w_n)$ and invoking \cite[Proposition 3]{MR3393266} we get 
\begin{equation}
\label{EqF10}
\int_{\mathbb{R}^N} |(-\Delta )^{\frac{s}{2}} T_k(w_n)|^2 \, dx \leq \int_{\Omega} \dfrac{K(x)}{w_n^q} T_k(w_n) \, dx + \int_{\Omega} \frac{f_n(x)}{(w_n+\frac{1}{n})^{\gamma}} T_k(w_n) \, dx + \int_{\Omega}T_k(w_n) \mu_n \, dx. 
\end{equation}
For the first term on the right-hand side of \eqref{EqF10}, we obtain the following estimate thanks to \eqref{ESTIM1} and \eqref{ESTIM2}.
\begin{align}
\int_{\Omega} \dfrac{K(x)}{w_n^q} T_k(w_n) \, dx \leq C_1 \|T_k(w_n)\|_{X_0^s(\Omega)}. \label{EqF13}
\end{align}
For the second term we have the following estimate:
\begin{equation}
\label{EqF13.5}
\int_{\Omega} \frac{f_n(x)}{(w_n+\frac{1}{n})^{\gamma}} T_k(w_n) \, dx \leq \int_{\Omega} \frac{f_n(x)T_k^{\gamma}(w_n)}{(w_n+\frac{1}{n})^{\gamma}} T_k^{1-\gamma}(w_n) \, dx \leq k^{1-\gamma} \|f\|_{L^1(\Omega)} =  C_2, 
\end{equation}
by noting that $ \frac{T_k^{\gamma}(w_n)}{(w_n+\frac{1}{n})^{\gamma}} \leq \frac{w_n^{\gamma}}{(w_n+\frac{1}{n})^{\gamma}} \leq 1$. Also for the last term:
\begin{equation}
\label{EqF14}
\int_{\Omega}T_k(w_n) \mu_n \, dx \leq k \|\mu\|_{L^1(\Omega)} = C_3.
\end{equation}
Thus from \eqref{EqF10}, \eqref{EqF13}, \eqref{EqF13.5} and \eqref{EqF14} we obtain
\begin{equation*}
\begin{aligned}
\int_{\mathbb{R}^N} |(-\Delta )^{\frac{s}{2}} T_k(w_n)|^2 \, dx & \leq C_1 \|T_k(w_n)\|_{X_0^s(\Omega)}+C_4,
\end{aligned}
\end{equation*}
or equivalently $\|T_k(w_n)\|_{X_0^s(\Omega)}^2 \leq C_1 \|T_k(w_n)\|_{X_0^s(\Omega)}+C_4$, which implies the boundedness of $\{T_k(w_n)\}_{n=1}^{\infty}$ in $X_0^s(\Omega)$.

\hfil \newline
\textbf{Second case: $\{T_k^{\frac{\gamma+1}{2}}(w_n)\}_{n=1}^{\infty}$ is bounded in $X_{0}^s(\Omega)$ and $\{T_k(w_n)\}_{n=1}^{\infty}$ is bounded in $X_{\mathrm{loc}}^s(\Omega)$.}

Taking $\phi=T_k^{\gamma}(w_n)$ as a test function in \eqref{Eq4.75} we obtain
\begin{equation}
\label{Eq10}
\int_{\mathbb{R}^N} (-\Delta )^{\frac{s}{2}} w_n (-\Delta )^{\frac{s}{2}} T_k^{\gamma}(w_n) \, dx = \int_{\Omega} \dfrac{K(x)}{w_n^q} T_k^{\gamma}(w_n) \, dx + \int_{\Omega} \frac{f_n(x)}{(w_n+\frac{1}{n})^{\gamma}} T_k^{\gamma}(w_n) \, dx + \int_{\Omega}T_k^{\gamma}(w_n) \mu_n \, dx.
\end{equation}
For the left-hand side, by using \eqref{Eq3.349n} and the following elementary inequality
\begin{equation}
\label{EqElementary}
(s_1-s_2) \Big(s_1^{\gamma}-s_2^{\gamma} \Big) \geq \frac{4\gamma}{(\gamma+1)^2} \Bigg(s_1^{\frac{\gamma+1}{2}}-s_2^{\frac{\gamma+1}{2}} \Bigg)^2, \qquad \forall s_1,s_2 \geq 0, \,\, \gamma>0,
\end{equation}
we get
\begin{align} \nonumber 
\int_{\mathbb{R}^N} (-\Delta )^{\frac{s}{2}} w_n (-\Delta )^{\frac{s}{2}} T_k^{\gamma}(w_n) \, dx  & =  \frac{C_{N,s}}{2} \iint_{D_{\Omega}} \frac{\Big(w_n(x)-w_n(y) \Big) \Big(T_k^{\gamma}(w_n)(x)-T_k^{\gamma}(w_n)(y) \Big)}{|x-y|^{N+2s}} \, dxdy \\\nonumber 
& \geq \frac{2 \gamma C_{N,s}}{(\gamma+1)^2} \iint_{D_{\Omega}} \frac{|T_k^{\frac{\gamma+1}{2}}w_n(x)-T_k^{\frac{\gamma+1}{2}}(w_n)(y)|^2}{|x-y|^{N+2s}} \, dxdy \\
& \geq C_5 \int_{\mathbb{R}^N} |(-\Delta )^{\frac{s}{2}} T_k^{\frac{\gamma+1}{2}}(w_n)|^2 \, dx. \label{Eq10.0001}
\end{align}
For the first term on the right-hand side of \eqref{Eq10}, we have the following estimate thanks to \eqref{ESTIM1} and \eqref{ESTIM2}.
\begin{align} \nonumber
\int_{\Omega} \dfrac{K(x)}{w_n^q} T_k^{\gamma}(w_n) \, dx & \leq k^{\frac{\gamma-1}{2}} \int_{\Omega} \dfrac{K(x) T_k^{\frac{\gamma+1}{2}}(w_n)}{w_n^q} \, dx \\
& \leq C_6 \|T_k^{\frac{\gamma+1}{2}}(w_n)\|_{X_0^s(\Omega)}. \label{Eq13}
\end{align}
Also for the other terms, we have
\begin{align} \label{Eq13.05} 
& \int_{\Omega} \frac{f_n(x)}{(w_n+\frac{1}{n})^{\gamma}} T_k^{\gamma}(w_n) \, dx \leq \int_{\Omega} \frac{f_n(x) w_n^{\gamma} }{(w_n+\frac{1}{n})^{\gamma}} \, dx \leq \|f\|_{L^1(\Omega)} =  C_7, \\ 
& \int_{\Omega}T_k^{\gamma}(w_n) \mu_n \, dx \leq k^\gamma \|\mu\|_{L^1(\Omega)} = C_8. \label{Eq14}
\end{align}
Thus from \eqref{Eq10}, \eqref{Eq10.0001}, \eqref{Eq13}, \eqref{Eq13.05}, and \eqref{Eq14} we obtain
\begin{equation*}
\label{Eq15}
\begin{aligned}
\int_{\mathbb{R}^N} |(-\Delta )^{\frac{s}{2}} T_k^{\frac{\gamma+1}{2}}(w_n)|^2 \, dx & \leq C_9 \|T_k^{\frac{\gamma+1}{2}}(w_n)\|_{X_0^s(\Omega)}+C_{10},
\end{aligned}
\end{equation*}
which implies the boundedness of $\{T_k^{\frac{\gamma+1}{2}}(w_n)\}_{n=1}^{\infty}$ in $X_0^s(\Omega)$.

Now we show that $\{T_k(w_n)\}_{n=1}^{\infty}$ is bounded in $X_{\mathrm{loc}}^s(\Omega)$. For this purpose first note that, the strong maximum principle provides that, for any compact set $K \Subset \Omega$, there exists $C(K)>0$ such that
\begin{equation}
\label{STRICTP}
w_n(x) \geq w_1(x) \geq C(K) >0, \qquad \text{a.e. in} \,\, K.
\end{equation}
Therefore
$$ T_k(w_n) \geq T_k(w_1) \geq \tilde{C}:=\min\{k, C(K) \}.$$
For $(x,y) \in K \times K$, define $\alpha_n:= \dfrac{T_k(w_n)(x)}{\tilde{C}}$, and $\beta_n:= \dfrac{T_k(w_n)(y)}{\tilde{C}}$.
Since $\alpha_n, \beta_n \geq 1$, we have the following estimate by applying an elementary inequality:
$$ (\alpha_n - \beta_n)^2 \leq \Bigg(\alpha_n^{\frac{\gamma+1}{2}}-\beta_n^{\frac{\gamma+1}{2}} \Bigg)^2. $$
Now by the definition of $\alpha_n$ and $\beta_n$, we obtain
$$ \Big(T_k(w_n(x)) - T_k(w_n(y)) \Big)^2 \leq \tilde{C}^{1-\gamma} \Bigg(T_k^{\frac{\gamma+1}{2}}w_n(x)-T_k^{\frac{\gamma+1}{2}}w_n(y) \Bigg)^2.$$
Thus we get the boundedness of $\{T_k(w_n)\}_{n=1}^{\infty}$ in $X_{\mathrm{loc}}^s(\Omega)$ by the boundedness of $\{T_k^{\frac{\gamma+1}{2}}(w_n)\}_{n=1}^{\infty}$ in $X_0^s(\Omega)$.
\end{proof}

Now, we have the following Proposition in the spirit of \cite[Theorem 23]{MR3393266} and \cite[Theorem 4.10]{MR3479207}.
\begin{proposition}
\label{Pro3}
\begin{enumerate}[leftmargin=*, label=(\roman*)]
\item
If $0<\gamma \leq 1$, then $\{w_n\}_{n=1}^{\infty}$ is bounded in $X_0^{s_1,p}(\Omega)$, for all $s_1<s$ and for all $p < \frac{N}{N-s}$.
\item
If $ \gamma> 1$, then $\{w_n\}_{n=1}^{\infty}$ is bounded in $X_{\mathrm{loc}}^{s_1,p}(\Omega)$, for all $s_1<s$ and for all $p < \frac{N}{N-s}$.
\end{enumerate}
\end{proposition} 
\begin{proof}

\hfil \newline
\textbf{First case: $0< \gamma \leq 1$.}

By following the proof of \cite[Proposition 3.4]{bayrami2020existence} we obtain that $\{ (-\Delta )^{\frac{s}{2}} w_n \}$ is bounded in $M^{\frac{N}{N-s}}(\Omega)$ and the embedding \eqref{Eq3.2} implies the boundedness of $\{ (-\Delta )^{\frac{s}{2}} w_n\}$ in $L^p(\Omega)$, for all $p < \frac{N}{N-s}$. Notice that since we have $N>2s$, therefore $\frac{N}{N-s} <2$. Now, by invoking Theorem 5 and Proposition 10 in chapter 5 of the reference book \cite{MR0290095}, we get the boundedness of $\{w_n\}$ in $X_0^{s_1,p}(\Omega)$, for all $s_1<s$ and for all $p<\frac{N}{N-s}$. (In \cite{MR0290095}, our function space $X_0^{s,p}(\Omega)$, reads as $\Lambda_{s}^{p,p}(\mathbb{R}^N)$, and $\mathcal{L}_{s}^p(\mathbb{R}^N)$ denotes the space of Bessel potentials, see \cite[subsection 3.2]{MR0290095}.)

\hfil \newline
\textbf{Second case: $ \gamma > 1$.}

By following the proof of \cite[Lemma 2.8]{MR3935063} and the similar reasoning as in the first case, i.e. $0< \gamma \leq 1$, the proof obtains.
\end{proof}
Now we are ready to prove the Theorem \ref{Thm1}.

\begin{proof}[Proof of Theorem \ref{Thm1}]
By Proposition \ref{Pro3} we deduce that there exists a function $u$ such that
\begin{itemize} 
\item
if $0<\gamma \leq 1$, then $u \in X_0^{s_1,p}(\Omega)$, for all $s_1<s$ and for all $p < \frac{N}{N-s}$, and up to a subsequence, $ w_n \to u$ weakly in $X_{0}^{s_1,p}(\Omega)$. Moreover, $T_k(u) \in X_0^s(\Omega)$.
\item
if $ \gamma > 1$, then $u \in X_{\mathrm{loc}}^{s_1,p}(\Omega)$, for all $s_1<s$ and for all $p < \frac{N}{N-s}$, and up to a subsequence, $ w_n \to u$ weakly in $X_{\mathrm{loc}}^{s_1,p}(\Omega)$. Moreover, $T_k^{\frac{\gamma+1}{2}}(u) \in X_0^s(\Omega)$, and $T_k(u) \in X_{\mathrm{loc}}^s(\Omega)$.
\end{itemize}
Besides, by using the embedding \eqref{Eq3.350}, up to a subsequence, we can assume that
\begin{itemize} 
\item
$ w_n \to u $ in $L^r(\Omega)$, for any $r \in [1,2_{s}^*)$, in the case $0<\gamma \leq 1$;
\item
$ w_n \to u $ in $L_{\mathrm{loc}}^r(\Omega)$, for any $r \in [1,2_{s}^*)$, in the case $\gamma> 1$;
\item
$w_n(x) \to u(x)$ pointwise a.e. in $\Omega$.
\end{itemize} 
Now for every fixed $\phi \in \mathcal{T}(\Omega)$, by noting that $\phi$ has compact support and using \eqref{STRICTP}, we could pass to the limit and obtain
$$
\begin{aligned}
& \lim_{n \to \infty} \int_{\Omega} \dfrac{K(x)}{w_n^q} \phi \, dx = \int_{\Omega} \dfrac{K(x)}{u^q} \phi \, dx,\\
& \lim_{n \to \infty} \int_{\Omega} \dfrac{f_n(x)}{(w_n+\frac{1}{n})^{\gamma}} \phi \, dx = \int_{\Omega} \dfrac{f(x)}{u^{\gamma}} \phi \, dx,\\
& \lim_{n \to \infty} \int_{\mathbb{R}^N} (-\Delta )^{\frac{s}{2}} w_n (-\Delta )^{\frac{s}{2}} \phi \, dx = \lim_{n \to \infty} \int_{\mathbb{R}^N} w_n (-\Delta)^{s} \phi \, dx = \int_{\mathbb{R}^N} u (-\Delta)^{s} \phi \, dx.
\end{aligned}
$$
Also trivially we have
$$ \int_{\Omega} \mu_n \phi \, dx \to \int_{\Omega} \mu \phi \, dx, \qquad \forall \phi \in \mathcal{T}(\Omega).$$
Since for every $K \Subset \Omega$, there exists $C_K>0$ such that $w_n(x) \geq w_1(x) \geq C_K$ a.e. in $K$ and also $w_n \equiv 0$ in $\big(\mathbb{R}^N \setminus \Omega \big)$ and because of $w_n(x) \to u(x)$ a.e. in $\Omega$, thus $u$ is a weak solution to problem \eqref{Eq1}. This completes the proof of the existence result. 

As mentioned before, for uniqueness in the case $0<\gamma \leq 1$, we would like to consider the entropy solution.  

First we prove the existence of an entropy solution. Let $0<\gamma \leq 1$ and consider the following approximation problem.
\begin{equation}
\label{Eq30}
\begin{cases}
(- \Delta)^s u_n = K(x)u_n^{-q}+ f_n(x) \big(u_n+\frac{1}{n} \big)^{-\gamma} + \mu_n &   \mathrm{in} \,\, \Omega,\\ u_n>0 &   \mathrm{in} \,\, \Omega, \\ u_n=0 &   \mathrm{in} \,\,  \big(\mathbb{R}^N \setminus \Omega \big).
\end{cases}
\end{equation}
For each fixed $k$ and each fixed $\phi \in X_0^s(\Omega) \cap L^{\infty}(\Omega)$, let consider the sequence $\{ T_k(u_n-\phi)\}_{n=1}^{\infty}$, which is a bounded sequence in $X_0^s(\Omega)$, by Proposition \ref{Pro2}. Therefore, up to a subsequence, $T_k(u_n-\phi) \to T_k(u-\phi)$ weakly in $X_0^s(\Omega)$, as $n \to \infty$, where $u$ is the weak solution to \eqref{Eq1}. Also, $\{ T_k(u_n-\phi)\}_{n=1}^{\infty}$ is an increasing sequence of non-negative functions by the strict monotonicity of the operator $(- \Delta)^s u - K(x) u^{-q}$ (a similar argument as in \cite[Lemmas 3.1]{MR3479207}) and the increasing behaviour of $f_n(x) \big(u_n+\frac{1}{n} \big)^{-\gamma} + \mu_n$. Once more the strict monotonicity of $(-\Delta)^s $, and $T_k(u_n-\phi) \uparrow T_k(u-\phi)$ implies that $ T_k(u_n-\phi) \to T_k(u-\phi)$ strongly in $X_0^s(\Omega)$ (see for example \cite[Lemmas 2.18]{MR3479207} for this compactness result). On the other hand, by using the estimates \eqref{ESTIM1}, \eqref{ESTIM2}, and \cite[Proposition 4]{MR3393266} we have
$$ 
\begin{aligned}
\int_{\Omega} \frac{K(x) |T_k(u_n-\phi) |}{u_1^{q}} \, dx & \leq C \Big\| |T_k(u_n-\phi) |\Big\|_{X_0^s(\Omega)} \\
& \leq C \|T_k(u_n-\phi)\|_{X_0^s(\Omega)}  \\
& \leq C_1<+ \infty, \qquad \text{uniformly in} \,\, n.
\end{aligned}
$$
Now Fatou's lemma and the monotonicity of $u_n$ gives
$$ \Bigg| \frac{K(x)T_k(u_n-\phi)}{u_n^{q}}\Bigg| \leq \frac{K(x) |T_k(u-\phi)|}{u_1^{q}} \in L^1(\Omega). $$
Also, we have
$$ \Bigg| \frac{f_n(x)T_k(u_n-\phi)}{(u_n+\frac{1}{n})^{\gamma}} \Bigg| \leq k \frac{f}{u_1^{\gamma}} \in L^1(\Omega), $$
because of \eqref{EqBNTBOU}, and the assumptions \eqref{EnASSUMPTIONS}. Finally, using $T_k(u_n-\phi)$ as a test function in \eqref{Eq30}, and passing to the limit we find an entropy solution even with the equalities instead of the inequalities in Definition \ref{DEFN}, i.e. \eqref{Eq3.500}.

Let $u$ and $v$ be two entropy solutions. Testing $u$ with $\phi=T_h(v)$ and $v$ with $\phi=T_h(u)$ in the weak formulation of entropy inequalities and following the idea of \cite[Section 5]{MR1354907}, the rest of the proof of uniqueness is almost immediate. More precisely, we have 
\begin{equation}
\label{ENTEq20}
\begin{aligned}
\int_{\{|u-T_h(v)|<k\}} (-\Delta)^{\frac{s}{2}} u (-\Delta)^{\frac{s}{2}} (u-T_h(v)) \, dx \leq & \int_{\Omega} \frac{K(x) T_k(u-T_h(v))}{u^{q}} \, dx +\int_{\Omega} \frac{f T_k(u-T_h(v))}{u^{\gamma}} \, dx \\
& \quad + \int_{\Omega} \mu T_k(u-T_h(v)) \, dx,
\end{aligned}
\end{equation}
and
\begin{equation}
\label{ENTEq21}
\begin{aligned}
\int_{\{|v-T_h(u)|<k\}} (-\Delta)^{\frac{s}{2}} v (-\Delta)^{\frac{s}{2}} (v-T_h(u)) \, dx \leq & \int_{\Omega} \frac{K(x) T_k(v-T_h(u))}{v^{q}} \, dx +\int_{\Omega} \frac{f T_k(v-T_h(u))}{v^{\gamma}} \, dx \\
& \quad + \int_{\Omega} \mu T_k(v-T_h(u)) \, dx.
\end{aligned}
\end{equation}
Adding up the left-hand sides of \eqref{ENTEq20} and \eqref{ENTEq21} and restricting them to 
$$ A_0^h=\{ x \in \Omega \,:\, |u-v|<k, \, |u|<h, \, |v|<h \}, $$
we have the following obvious estimate
\begin{equation}
\label{ENTEq22}
\int_{A_0^h} |(-\Delta)^{\frac{s}{2}} (u-v)|^2 \, dx \geq 0.
\end{equation}
Also, summing the right-hand sides of \eqref{ENTEq20} and \eqref{ENTEq21} when restricted to $A^h_0$ gives 
\begin{equation}
\label{ENTEq23}
\int_{A_0^h} K(x)(u-v)(u^{-q}-v^{-q}) \, dx + \int_{A_0^h} f(u-v)(u^{-\gamma}-v^{-\gamma}) \, dx \leq 0.
\end{equation}
Now, consider the set $ A_1^h=\{ x \in \Omega \,:\, |u-T_h(v)|<k, \, |v| \geq h \}$. 
When restricted to $A_1^h$, we have the following for the left-hand side of \eqref{ENTEq20}:
\begin{equation}
\label{ENTEq24}
\int_{A_1^h} |(-\Delta)^{\frac{s}{2}} u|^2 \, dx \geq 0.
\end{equation}
On the other hand, when restricted to $A_1^h$, the right-hand side of \eqref{ENTEq20} is
\begin{equation}
\label{ENTEq24+++}
\int_{A_1^h} K(x)u^{-q} (u-h) \, dx +\int_{A_1^h} f u^{-{\gamma}} (u-h) \, dx+ \int_{A_1^h} \mu (u-h) \, dx,
\end{equation}
which goes to zero as $h \to \infty$.

Finally on the remaining set $ A_2^h=\{ x \in \Omega \,:\, |u-T_h(v)|<k, \, |v| < h, \, |u| \geq h \}$, the left-hand side of \eqref{ENTEq20} is as follows
\begin{equation}
\label{ENTEq25}
\int_{A_2^h} (-\Delta)^{\frac{s}{2}} u (-\Delta)^{\frac{s}{2}} (u-v) \, dx,
\end{equation}
which goes to zero as $h \to \infty$.

The right-hand side of \eqref{ENTEq20}, when restricted to $A_2^h$, is as follows
\begin{equation}
\label{ENTEq25++}
\begin{aligned}
\int_{A_2^h} K(x)u^{-q} (u-v) \, dx + \int_{A_2^h} fu^{-\gamma} (u-v) \, dx + \int_{A_2^h} \mu (u-v) \, dx,
\end{aligned}
\end{equation}
which also goes to zero as $h \to \infty$.

Similarly, we can estimate the left-hand side of \eqref{ENTEq21} on the sets $B_1^h=\{x \in \Omega \,:\, |v-T_h(u)|<k, \, |u| \geq h \}$ and $B_2^h=\{  x \in \Omega \,:\, |v-T_h(u)|<k, \, |u| < h, \, |v| \geq h  \}$ and find that
\begin{equation}
\label{ENTEq26}
\int_{B_1^h} |(-\Delta)^{\frac{s}{2}} v|^2 \, dx \geq 0,
\end{equation}
and
\begin{equation}
\label{ENTEq27}
\int_{B_2^h} (-\Delta)^{\frac{s}{2}} v (-\Delta)^{\frac{s}{2}} (v-u) \, dx \to 0, \qquad \text{as} \,\, h \to 0.
\end{equation}

On the other hand for the right-hand side of \eqref{ENTEq21} on the sets $B_1^h=\{x \in \Omega \,:\, |v-T_h(u)|<k, \, |u| \geq h \}$ and $B_2^h=\{  x \in \Omega \,:\, |v-T_h(u)|<k, \, |u| < h, \, |v| \geq h  \}$, we have: 
\begin{equation}
\label{ENTEq26++}
\int_{B_1^h} K(x)v^{-q} (v-h) \, dx + \int_{B_1^h} fv^{-\gamma} (v-h) \, dx +\int_{B_1^h} \mu (v-h) \, dx \to 0, \quad \text{as} \,\, h \to 0,
\end{equation}
and
\begin{equation}
\label{ENTEq27++}
\int_{B_2^h} K(x)v^{-q} (v-u) \, dx + \int_{B_2^h} fv^{-\gamma} (v-u) \mu \, dx + \int_{B_2^h} \mu (v-u) \, dx \to 0, \quad \text{as} \,\, h \to 0.
\end{equation}
Putting all the estimates \eqref{ENTEq22}, \eqref{ENTEq23}, \eqref{ENTEq24}, \eqref{ENTEq24+++}, \eqref{ENTEq25}, \eqref{ENTEq25++}, \eqref{ENTEq26}, \eqref{ENTEq27}, \eqref{ENTEq26++}, and \eqref{ENTEq27++} together we obtain
$$ \int_{A_0^h} |(-\Delta)^{\frac{s}{2}} (u-v)|^2 \, dx \leq \mathrm{o}(h), \qquad \text{as} \,\, h \to 0. $$
Now, since $A_0^h$ goes to $\{ |u-v|<k \}$, as $h \to 0$ we have
$$ \int_{\{|u-v|<k \}} |(-\Delta)^{\frac{s}{2}} (u-v)|^2 \, dx \leq 0, \qquad \forall k. $$
Therefore $u \equiv v$, and the uniqueness is proved.
\end{proof}

\begin{remark}
\label{REM1}
The above proof for the existence of a positive weak solution can be generalized easily in the following ways. 
\begin{itemize} 
\item
$\mu$ can be a non-negative bounded Radon measure in $\Omega$. More precisely, in the approximating problem \eqref{Eq4.75}, instead of $\mu_n=T_n(\mu)$, we should consider $\mu_n$ as a sequence of non-negative $L^{\infty}(\Omega)$ functions bounded in $L^1(\Omega)$ such that converges to $\mu$ in the narrow topology of measures, i.e. 
$$
\int_{\Omega} \phi \mu_n \, dx \to \int_{\Omega} \phi  \, d\mu, \qquad \forall \phi \in C_c(\Omega). 
$$
Now the same proof can be performed.
\item
Also, instead of $f \in L^1(\Omega)$, we can consider $ f \in L^1(\Omega) + X^{-s}(\Omega)$, if the datum $\mu$ is regular enough to give the result that the solution $u$ obtained above satisfies $u \in X_{\mathrm{loc}}^s(\Omega)$. See Remark \ref{REMM01}. In this case, in the approximating problem \eqref{Eq4.75}, instead of $f_n=T_n(f)$, we should consider $f_n$ as a sequence of non-negative $X_0^s(\Omega) \cap L^{\infty}(\Omega)$ functions bounded in $L^1(\Omega)$ such that converges to $f$ in the following sense: 
$$
\int_{\Omega} \phi f_n \, dx \to \big\langle \phi, f \big\rangle_{X_0^s(\Omega) \cap L^{\infty}(\Omega), L^1(\Omega) + X^{-s}(\Omega)}, \qquad \forall \phi \in X_0^s(\Omega) \cap L^{\infty}(\Omega). 
$$
Here $\big\langle \cdot, \cdot \big\rangle_{X_0^s(\Omega) \cap L^{\infty}(\Omega), L^1(\Omega) + X^{-s}(\Omega)}$ denotes the duality pairing between $X_0^s(\Omega) \cap L^{\infty}(\Omega)$, and $L^1(\Omega) + X^{-s}(\Omega)$. Notice that $L^1(\Omega)+X^{-s}(\Omega)$ embeds in the dual space of $X_0^s(\Omega)\cap L^{\infty}(\Omega)$. Besides, see section \ref{Section4} for some relaxation on the assumption on $f$, in order to prove the existence of solutions.
\end{itemize}
\end{remark}

\begin{remark}
\label{RemEntropy}
It is important to note that for the existence of a positive entropy solution, our argument needs that the approximating terms $f_n$, and $\mu_n$ be the increasing ones. This is always possible by the usual truncation technique for $\mu \in L^1(\Omega)$. Generally, non-negative bounded Radon measures cannot be approximated by an increasing sequence of bounded functions, see \cite{MR2592976}. This increasing behaviour of the approximations is the key property in the proof of our existence result for the entropy solution. At the end of section 4, we will see that for the measure $\mu \in L^1(\Omega) + X^{-s}(\Omega)$, it is possible to approximate it by an increasing sequence of bounded functions.
\end{remark}

\begin{proof}[Proof of Theorem \ref{Thm2}]
\hfil \newline
\textbf{Uniqueness for the first case: $ \gamma > 1$.}

Let $u_1$ and $u_2$ be two positive weak solutions to problem \eqref{Eq1}. We know that $T_k^{\frac{\gamma+1}{2}}(u_i) \in X_0^s(\Omega)$, for $i=1,2$. Now define $w=u_1-u_2$. Then we have
\begin{equation}
\label{Uniquness1}
\int_{\mathbb{R}^N} w (-\Delta)^s \phi \, dx = \int_{\Omega} \Bigg(\frac{K(x)}{u_1^q}-\frac{K(x)}{u_2^q} \Bigg) \phi \, dx + \int_{\Omega} \Bigg(\frac{f(x)}{u_1^{\gamma}}-\frac{f(x)}{u_2^{\gamma}} \Bigg) \phi \, dx, \quad \forall \phi \in \mathcal{T}(\Omega).
\end{equation}
By using the comparison principle and thanks to the estimates \eqref{ESTIM1} and \eqref{ESTIM2} and also a density argument, the equality \eqref{Uniquness1} holds for all $\phi \in X_0^s(\Omega) \cap L^{\infty}(\Omega)$ if we further assume that the assumption \eqref{ASSUMPTIONS} holds. More precisely, it is enough to show that in the second term on the right-hand side of \eqref{Uniquness1}, any $\phi \in X_0^s(\Omega) \cap L^{\infty}(\Omega)$ can be used as a test function. By using the comparison principle, H\"older inequality, the estimates \eqref{ESTIM1} and \eqref{ESTIM2} and the assumptions \eqref{ASSUMPTIONS} on $f$ this follows quickly.

Using $T_k^{\frac{\gamma+1}{2}}(w^-)$ as a test function in \eqref{Uniquness1} and invoking \eqref{Eq3.349n} and \eqref{EqElementary} we deduce that
$$ 
\int_{\mathbb{R}^N} |(-\Delta)^{\frac{s}{2}} T_k^{\frac{\gamma+3}{4}}(w^-)|^2 \, dx \leq \frac{(\gamma+3)^2}{4(\gamma+1) C_{N,s}} \Bigg\{ \int_{\Omega} \Bigg(\frac{K(x)}{u_1^q}-\frac{K(x)}{u_2^q} \Bigg) T_k^{\frac{\gamma+1}{2}}(w^-) \, dx + \int_{\Omega} \Bigg(\frac{f(x)}{u_1^{\gamma}}-\frac{f(x)}{u_2^{\gamma}} \Bigg) T_k^{\frac{\gamma+1}{2}}(w^-) \, dx \Bigg \} \leq 0.
$$
Therefore, $T_k^{\frac{\gamma+3}{4}}(w^-) \equiv 0$. So we reach at the conclusion that $T_k^{\frac{\gamma+3}{4}}(u_1) \geq T_k^{\frac{\gamma+3}{4}}(u_2)$. Similar argument shows that $T_k^{\frac{\gamma+3}{4}}(u_1) \leq T_k^{\frac{\gamma+3}{4}}(u_2)$. Therefore $T_k^{\frac{\gamma+3}{4}}(u_1)=T_k^{\frac{\gamma+3}{4}}(u_2)$, for any $k$, and the uniqueness follows.

\hfil \newline
\textbf{Uniqueness for the second case: $ 0 < \gamma \leq 1$.}

Let $u_1$ and $u_2$ be two positive weak solutions to problem \eqref{Eq1} and define $w=u_1-u_2$. We know that $T_k(w) \in X_0^s(\Omega)$. Using $T_k(w^-)$ as a test function in \eqref{Uniquness1} we deduce that
$$
\int_{\mathbb{R}^N} |(-\Delta)^{\frac{s}{2}} T_k(w^-)|^2 \, dx \leq \int_{\Omega} \Bigg(\frac{K(x)}{u_1^q}-\frac{K(x)}{u_2^q} \Bigg) T_k(w^-) \, dx + \int_{\Omega} \Bigg(\frac{f(x)}{u_1^{\gamma}}-\frac{f(x)}{u_2^{\gamma}} \Bigg) T_k(w^-) \, dx \leq 0.
 $$
Therefore, $T_k(w^-) \equiv 0$. Similarly $T_k(w^+) \equiv 0$, for any $k$, and the uniqueness follows. Therefore the notion of the entropy solution and the weak solution coincides in this case.
\end{proof}

\section{Some relaxation assumptions on a datum to prove the existence of solutions}
\label{Section4}
This section includes approaches that are the same as the procedures used in \cite{MR2592976}. Specifically, the idea of the proofs is the same, while the techniques we deal with the fractional Laplacian instead of the Laplacian operator. We will relax the assumption on $f$ to obtain the existence results. Before continuing, we need to recall the concept of capacity (see \cite{MR3409135, MR3890060, MR2777530} for the classical notion of capacity associated with $W^{1,p}(\mathbb{R}^N)$). More precisely, what we need is the concept of fractional capacity, which will be defined below. See \cite{MR3306694, MR3518675} for more information. The capacity theory allows the study of small sets in $\mathbb{R}^N$. One can show that in $\mathbb{R}^N$, there are zero Lebesgue sets with capacity strictly bigger than zero. So, it makes sense to speak about the values of a function $u \in X^{s,2}(\Omega)$ on a set $E \subset \Omega$ with capacity bigger than zero and study its fine properties.

Let $K$ be a compact subset of $\Omega$. The $(s,2)$-capacity of the compact set $K$ defines as follows: 
$$ \mathrm{Cap}_{s,2}(K)=\inf \Bigg\{ \int_{\mathbb{R}^N} |(-\Delta)^{\frac{s}{2}} \Psi|^2 \, dx, \quad \Psi \in C^{\infty}_c(\Omega), \,\, \Psi \geq \chi_K \Bigg\}, $$
where $\chi_K$ represents the characteristic function of $K$. The above definition of capacity extends naturally to all Borel subsets of $\Omega$ by the regularity.

The following proposition illuminates our discussion of the relaxation theorem that will come later in Theorem \ref{Relaxation-Theorem}. See \cite{MR2989693} for a similar result in the case of the $p$-Laplacian operator with Hardy potential.
\begin{proposition}
\label{ReThm1}
Let $\nu$ be a non-negative Radon measure concentrated on a Borel set $E$ of zero $(s,2)$-capacity, and let $g_n$ be a sequence of non-negative $L^{\infty}(\Omega)$ functions that converges to $\nu$ in the narrow topology of measures. Let $u_n$ be the solution to
\begin{equation}
\label{ReEq1}
\begin{cases}
(-\Delta )^s u_n = \dfrac{g_n(x)}{(u_n+\frac{1}{n})^{\gamma}} & \mathrm{in} \,\, \Omega,\\ u_n>0 & \mathrm{in} \,\, \Omega, \\ u_n=0 & \mathrm{in} \,\,  \big(\mathbb{R}^N \setminus \Omega \big).
\end{cases}
\end{equation}
\begin{enumerate}[leftmargin=*, label=(\roman*)]
\item
If $ \gamma = 1$, then $u_n$ converges weakly to zero in $X_0^s(\Omega)$.
\item
If $ \gamma > 1$, then $u_n^{\frac{\gamma+1}{2}}$ converges weakly to zero in $X_0^s(\Omega)$.
\end{enumerate}
\end{proposition}

\begin{proof} 
We follow \cite[Theorem 3.4 and Theorem 4.5]{MR2592976}. 
\hfil \newline
\textbf{First case: $ \gamma = 1$.}

Since $g_n$ is bounded in $L^1(\Omega)$, a similar argument as in the proof of Proposition \ref{Pro2} shows that $u_n$ is bounded in $X_0^s(\Omega)$. Therefore, up to a subsequence, it converges weakly in $X_0^s(\Omega)$, and a.e. in $\Omega$, to some function $u$. Using the definition of the $(s,2)$-capacity and noting that the set $E$ where the measure is concentrated has zero $(s,2)$-capacity, we can find that for any $\eta>0$; there exists $\Psi_{\eta} \in C_c^1(\Omega)$ such that
\begin{equation}
\label{ReEq2}
0 \leq \Psi_{\eta} \leq 1, \qquad 0 \leq \int_{\Omega} (1- \Psi_{\eta}) \, d\nu \leq \eta, \qquad \int_{\mathbb{R}^N}  |(-\Delta)^{\frac{s}{2}} \Psi_{\eta}|^2 \, dx \leq \eta.
\end{equation}
Since $g_n$ converges to $\nu$ in the narrow topology of measures; one has from \eqref{ReEq2} that 
\begin{equation}
\label{ReEq3}
\lim_{n \to + \infty} \int_{\Omega} g_n(1- \Psi_{\eta}) \, dx = \int_{\Omega} (1- \Psi_{\eta}) \, d\nu \leq \eta.
\end{equation}
Choosing $T_k(u_n) \big(1- \Psi_{\eta} \big)$ as test function in \eqref{ReEq1} and using \eqref{Eq3.349n} we deduce 
\begin{equation}
\label{ReEqT01}
\frac{C_{N,s}}{2} \iint_{D_{\Omega}} \frac{\big(u_n(x)-u_n(y) \big) \Big(T_k\big(u_n\big)(x) \big(1-\Psi_{\eta}(x) \big)-T_k\big(u_n\big)(y) \big(1-\Psi_{\eta}(y) \big) \Big)}{|x-y|^{N+2s}} \, dxdy = \int_{\Omega} \frac{g_n(x)}{u_n+\frac{1}{n}} T_k(u_n) \big(1- \Psi_{\eta} \big) \, dx. 
\end{equation}
For the right-hand side, noting that $\dfrac{T_k(u_n)}{u_n+\frac{1}{n}} \leq 1$, we have
$$ \int_{\Omega} \frac{g_n(x)}{u_n+\frac{1}{n}} T_k(u_n) \big(1- \Psi_{\eta}\big) \, dx \leq \int_{\Omega} g_n (1- \Psi_{\eta}) \, dx, $$
so that by \eqref{ReEq3},
\begin{equation}
\label{ReEqTT00}
0 \leq \limsup_{n \to + \infty} \int_{\Omega} \frac{g_n(x)}{u_n+\frac{1}{n}} T_k(u_n) (1- \Psi_{\eta}) \, dx \leq \eta.
\end{equation}
For the left-hand side of \eqref{ReEqT01} first of all, we consider the following elementary but tedious calculations.
\begin{equation}
\label{ReEqT1}
\begin{aligned}
& \big(u_n(x)-u_n(y) \big) \Big(T_k\big(u_n\big)(x) \big(1-\Psi_{\eta}(x)\big)-T_k\big(u_n\big)(y)\big(1-\Psi_{\eta}(y)\big) \Big) \\
& \qquad \qquad = \big(u_n(x)-u_n(y) \big) \Big(T_k\big(u_n\big)(x)-T_k\big(u_n\big)(y) \Big) \\
& \qquad \qquad \quad + \big(u_n(x)-u_n(y) \big) \Big(T_k\big(u_n\big)(y) \Psi_{\eta}(y)-T_k\big(u_n\big)(x) \Psi_{\eta}(x) \Big).
\end{aligned}
\end{equation}
For the second term on the right-hand side of \eqref{ReEqT1} we have
\begin{equation}
\label{ReEqT2}
\begin{aligned}
& \big(u_n(x)-u_n(y) \big) \Big(T_k \big(u_n \big)(y) \Psi_{\eta}(y)-T_k \big(u_n) \big(x)\Psi_{\eta}(x) \Big) \\
& \qquad \qquad = T_k(u_n)(y) \big(u_n(x)-u_n(y) \big) \big( \Psi_{\eta}(y)-\Psi_{\eta}(x) \big) \\
& \qquad \qquad \quad + \Psi_{\eta}(x) \big(u_n(x)-u_n(y) \big) \big( T_k(u_n)(y)-T_k(u_n)(x) \big). 
\end{aligned}
\end{equation}
Therefore from \eqref{ReEqT1} and \eqref{ReEqT2} we obtain
$$ 
\begin{aligned}
& \big(u_n(x)-u_n(y) \big) \Big(T_k\big(u_n\big)(x) \big(1-\Psi_{\eta}(x)\big) - T_k\big(u_n\big)(y)\big(1-\Psi_{\eta}(y)\big) \Big) \\
& \qquad \qquad = \big(u_n(x)-u_n(y) \big) \Big(T_k\big(u_n\big)(x)-T_k\big(u_n\big)(y) \Big) \\
& \qquad \qquad \quad + \Psi_{\eta}(x) \big(u_n(x)-u_n(y) \big) \big( T_k(u_n)(y)-T_k(u_n)(x) \big) \\
& \qquad \qquad \quad + T_k(u_n)(y) \big(u_n(x)-u_n(y) \big) \big( \Psi_{\eta}(y)-\Psi_{\eta}(x) \big). 
\end{aligned}
$$ 
Invoking this equality in the left-hand side of \eqref{ReEqT01} gives
\begin{align} \nonumber
& \iint_{D_{\Omega}}  \frac{\big(u_n(x)-u_n(y) \big) \Big(T_k(u_n)(x)(1-\Psi_{\eta}(x)) -T_k(u_n)(y)(1-\Psi_{\eta}(y)) \Big)}{|x-y|^{N+2s}} \, dxdy \\ \nonumber
&  \quad = \iint_{D_{\Omega}} \frac{\big(u_n(x)-u_n(y) \big) \Big(T_k\big(u_n\big)(x)-T_k\big(u_n\big)(y) \Big)}{|x-y|^{N+2s}} \, dxdy \\ \label{ReEqTT001}
&  \qquad +  \iint_{D_{\Omega}} \frac{\Psi_{\eta}(x) \big(u_n(x)-u_n(y) \big) \big( T_k(u_n)(y)-T_k(u_n)(x) \big)}{|x-y|^{N+2s}} \, dxdy \\ \nonumber
&  \qquad + \iint_{D_{\Omega}} \frac{T_k(u_n)(y) \big(u_n(x)-u_n(y) \big) \big( \Psi_{\eta}(y)-\Psi_{\eta}(x) \big)}{|x-y|^{N+2s}} \, dxdy. 
\end{align}
For the first and second terms in the right-hand side of \eqref{ReEqTT001}, by letting $n \to + \infty$, and applying Fatou's lemma, we obtain
\begin{equation}
\label{ReEqTT11}
\begin{aligned}
& \iint_{D_{\Omega}} \frac{\big(u(x)-u(y) \big) \Big(T_k\big(u\big)(x)-T_k\big(u\big)(y) \Big)}{|x-y|^{N+2s}} \, dxdy \\
& \quad \leq \liminf_{n \to +\infty}\iint_{D_{\Omega}} \frac{\big(u_n(x)-u_n(y) \big) \Big(T_k\big(u_n\big)(x)-T_k\big(u_n\big)(y) \Big)}{|x-y|^{N+2s}} \, dxdy,
\end{aligned}
\end{equation}
and
\begin{equation}
\label{ReEqTT22}
\begin{aligned}
& \iint_{D_{\Omega}} \frac{\Psi_{\eta}(x) \big(u(x)-u(y) \big) \big( T_k(u)(y) - T_k(u)(x) \big)}{|x-y|^{N+2s}} \, dxdy \\ 
& \quad \leq \liminf_{n \to + \infty} \iint_{D_{\Omega}} \frac{\Psi_{\eta}(x) \big(u_n(x)-u_n(y) \big) \big( T_k(u_n)(y)-T_k(u_n)(x) \big)}{|x-y|^{N+2s}} \, dxdy.
\end{aligned}
\end{equation}
Now by considering \eqref{ReEqT01} and \eqref{ReEqTT001} and using \eqref{ReEqTT00}, \eqref{ReEqTT11} and \eqref{ReEqTT22} we obtain
\begin{equation}
\label{ReEqTT22.55}
\begin{aligned}
& \iint_{D_{\Omega}} \frac{\big(u(x)-u(y) \big) \Big(T_k\big(u\big)(x)-T_k\big(u\big)(y) \Big)}{|x-y|^{N+2s}} \, dxdy \\
& + \iint_{D_{\Omega}} \frac{\Psi_{\eta}(x) \big(u(x)-u(y) \big) \big( T_k(u)(y) - T_k(u)(x) \big)}{|x-y|^{N+2s}} \, dxdy \\
& +\limsup_{n \to + \infty} \iint_{D_{\Omega}} \frac{T_k(u_n)(y) \big(u_n(x)-u_n(y) \big) \big( \Psi_{\eta}(y)-\Psi_{\eta}(x) \big)}{|x-y|^{N+2s}} \, dx dy \leq \eta.
\end{aligned}
\end{equation}
For the integral under the limit in \eqref{ReEqTT22.55}, by using \eqref{Eq3.349n}, we have
\begin{equation}
\label{ReEqTT22.56}
\Bigg| \iint_{D_{\Omega}} \frac{T_k(u_n)(y) \big(u_n(x)-u_n(y) \big) \big( \Psi_{\eta}(y)-\Psi_{\eta}(x) \big)}{|x-y|^{N+2s}} \, dx dy \Bigg| \leq k \Bigg| \int_{\mathbb{R}^N} (-\Delta)^{\frac{s}{2}} u_n (-\Delta)^{\frac{s}{2}} \Psi_{\eta} \, dx \Bigg| \leq k \|u_n\|_{X_0^s(\Omega)} \| \Psi_{\eta} \|_{X_0^s(\Omega)}.
\end{equation}
Therefore, by the boundedness of $u_n$ in $X_0^s(\Omega)$, and also \eqref{ReEqTT22.56} and \eqref{ReEqTT22.55} we get
\begin{equation}
\label{ReEqTT22.557}
\begin{aligned}
& \iint_{D_{\Omega}} \frac{\big(u(x)-u(y) \big) \Big(T_k\big(u\big)(x)-T_k\big(u\big)(y) \Big)}{|x-y|^{N+2s}} \, dxdy \\
& \quad + \iint_{D_{\Omega}} \frac{\Psi_{\eta}(x) \big(u(x)-u(y) \big) \big( T_k(u)(y) - T_k(u)(x) \big)}{|x-y|^{N+2s}} \, dxdy \\
& \leq  \eta-C \| \Psi_{\eta} \|_{X_0^s(\Omega)},
\end{aligned}
\end{equation}
for some $C>0$. Now letting $\eta \to 0^+$ in \eqref{ReEqTT22.557}, and noting that, by \eqref{ReEq3}, $\Psi_{\eta} \to 0$ strongly in $X_0^s(\Omega)$ and also up to a subsequence, a.e. in $\Omega$, we get
\begin{equation}
\label{ReEq3.05}
\iint_{D_{\Omega}} \frac{\big(u(x)-u(y) \big) \Big(T_k\big(u\big)(x)-T_k\big(u\big)(y)\Big)}{|x-y|^{N+2s}} \, dxdy \leq 0.
\end{equation}
On the other hand, by using \eqref{Eq3.349n} and \cite[Proposition 4]{MR3393266}, or \eqref{EqElementary}, we get
\begin{align} \nonumber
\int_{\mathbb{R}^N} |(-\Delta)^{\frac{s}{2}} T_k(u)|^2 \, dx  & = \frac{C_{N,s}}{2} \iint_{D_{\Omega}} \frac{|T_k\big(u\big)(x)-T_k\big(u\big)(y)|^2}{|x-y|^{N+2s}} \, dxdy \\ \label{ReEqT3}
& \leq \frac{C_{N,s}}{2} \iint_{D_{\Omega}} \frac{\big(u(x)-u(y) \big) \Big(T_k\big(u\big)(x)-T_k\big(u\big)(y)\Big)}{|x-y|^{N+2s}} \, dxdy.
\end{align}
Thus \eqref{ReEq3.05} and \eqref{ReEqT3} implies that $T_k(u)=0$, and the proof for the case $\gamma=1$ is complete.

\hfil \newline
\textbf{Second case: $ \gamma > 1$.}

By the boundedness of $g_n$ in $L^1(\Omega)$, a similar argument as in the proof of Proposition \ref{Pro2} shows that $u_n^{\frac{\gamma+1}{2}}$ is bounded in $X_0^s(\Omega)$. Therefore, up to a subsequence, it converges weakly in $X_0^s(\Omega)$, and a.e. in $\Omega$, to some function $u$. Again, let $\Psi_{\eta} \in C_c^1(\Omega)$ be as in \eqref{ReEq2}. Testing \eqref{ReEq1} with $T_k^{\gamma}(u_n) \big(1- \Psi_{\eta}\big) $, and using \eqref{Eq3.349n} we obtain
$$
\frac{C_{N,s}}{2} \iint_{D_{\Omega}} \frac{\big(u_n(x)-u_n(y) \big) \Big(T_k^{\gamma}\big(u_n\big)(x) \big(1-\Psi_{\eta}(x) \big)-T_k^{\gamma} \big(u_n \big)(y) \big(1-\Psi_{\eta}(y) \big) \Big)}{|x-y|^{N+2s}} \, dxdy = \int_{\Omega} \frac{g_n(x)}{(u_n+\frac{1}{n})^{\gamma}} T_k^{\gamma}(u_n) \big(1- \Psi_{\eta} \big) \, dx. 
$$
For the right-hand side, since $\dfrac{T_k^{\gamma}(u_n)}{(u_n+\frac{1}{n})^{\gamma}} \leq 1$, we get
$$ \int_{\Omega} \frac{g_n(x)}{(u_n+\frac{1}{n})^{\gamma}} T_k^{\gamma}(u_n) (1- \Psi_{\eta}) \, dx \leq \int_{\Omega} g_n (1- \Psi_{\eta}) \, dx, $$
so that by \eqref{ReEq3}, we obtain
$$ 0 \leq \limsup_{n \to + \infty} \int_{\Omega} \frac{g_n(x)}{(u_n+\frac{1}{n})^{\gamma}} T_k^{\gamma}(u_n) \big(1- \Psi_{\eta} \big) \, dx \leq \eta. $$
Similar to the previous case, by letting $\eta \to 0^+$, $n \to + \infty$, we obtain
$$ \frac{C_{N,s}}{2} \iint_{D_{\Omega}} \frac{\big(u(x)-u(y) \big) \big(T_k^{\gamma}(u)(x)-T_k^{\gamma}(u)(y) \big)}{|x-y|^{N+2s}} \, dxdy \leq 0.$$
Now \eqref{Eq3.349n} together with the inequality \eqref{EqElementary} gives
$$ \int_{\mathbb{R}^N} |(-\Delta)^{\frac{s}{2}} T_k^{\frac{\gamma+1}{2}}(u)|^2 \, dx =0,$$
which is the desired result.
\end{proof}

In the next theorem, we will see that if we approximate a datum of the form $f+\nu$, with $f \in L^1(\Omega)$, and $\nu$ a singular measure with respect to $(s,2)$-capacity, only the $L^1$ term remains in the equation after the limiting procedure and the singular term $\nu$ "disappears."

\begin{theorem}[Relaxation Theorem]
\label{Relaxation-Theorem}
Let $f$ be a non-negative function in $L^1(\Omega)$, and let $f_n = T_n(f)$. Let $\nu$ and $g_n$ be as in Theorem \ref{ReThm1}. Also, let $u_n$ be the sequence of solutions to the following problems.
\begin{equation}
\label{ReEq4}
\begin{cases}
(-\Delta )^s u_n = \dfrac{K(x)}{u_n^{q}} + \dfrac{f_n+g_n}{(u_n+\frac{1}{n})^{\gamma}}+\mu_n & \mathrm{in} \,\, \Omega,\\ u_n>0 & \mathrm{in} \,\, \Omega, \\ u_n=0 & \mathrm{in} \,\,  \big(\mathbb{R}^N \setminus \Omega \big).
\end{cases}
\end{equation}
\begin{enumerate}[leftmargin=*, label=(\roman*)]
\item
If $ 0<\gamma \leq 1$, then $u_n$ converges to a function $u$, weakly in $ X_0^{s_1,p}(\Omega)$, for all $s_1<s$ and for all $p<\frac{N}{N-s}$, where $u$ is the weak solution to \eqref{Eq1} with datum $f$.
\item
If $ \gamma > 1$, then $u_n$ converges to a function $u$, weakly in $ X_{\mathrm{loc}}^{s_1,p}(\Omega)$, for all $s_1<s$ and for all $p<\frac{N}{N-s}$, where $u$ is the weak solution to \eqref{Eq1} with datum $f$.
\end{enumerate}
\end{theorem}
\begin{proof} We only proof the first case $0<\gamma \leq1$. The second case can be handled in the same way.

Let $0<\gamma \leq1$. Once again, since $f_n +g_n$ is bounded in $L^1(\Omega)$, the same proofs as in Proposition \ref{Pro2} and Proposition \ref{Pro3} shows that $\{u_n\}$ is bounded in $X_0^{s_1,p}(\Omega)$, for all $s_1<s$ and for all $p<\frac{N}{N-s}$. Hence, up to a subsequence, it converges weakly in $X_0^{s_1,p}(\Omega)$ to some function $u$. Since $g_n \geq 0$, then we have the following inequality in the weak sense:
$$ (-\Delta)^s u_n = \dfrac{K(x)}{u_n^{q}} + \dfrac{f_n+g_n}{(u_n+\frac{1}{n})^{\gamma}}+\mu_n \geq \dfrac{f_n}{(u_n+\frac{1}{n})^{\gamma}}. $$
Now let $v_n$ be the solution to the following problem:
\begin{equation}
\label{ReEq5}
\begin{cases}
(-\Delta )^s v_n = \dfrac{f_n}{(v_n+\frac{1}{n})^{\gamma}} & \mathrm{in} \,\, \Omega,\\ v_n>0 & \mathrm{in} \,\, \Omega, \\ v_n=0 & \mathrm{in} \,\,  \big(\mathbb{R}^N \setminus \Omega \big).
\end{cases}
\end{equation}
The comparison principle implies that $u_n \geq v_n$. Let $\Psi_{\eta} \in C_c^1(\Omega)$ as in \eqref{ReEq2} and fix a $\phi \in \mathcal{T}(\Omega)$. Choosing $\phi (1- \Psi_{\eta})$ as a test function in \eqref{ReEq4} we have
\begin{equation}
\label{ReEq6}
\int_{\mathbb{R}^N} (-\Delta)^{\frac{s}{2}} u_n (-\Delta)^{\frac{s}{2}}[\phi (1- \Psi_{\eta})] \, dx = \int_{\Omega} \frac{K(x) \phi (1- \Psi_{\eta})}{u_n^q}\, dx  + \int_{\Omega} \frac{(f_n+g_n)\phi (1- \Psi_{\eta})}{(u_n+\frac{1}{n})^{\gamma}}\, dx + \int_{\Omega} \phi (1- \Psi_{\eta}) \mu_n \, dx.
\end{equation}
Since in problem \eqref{ReEq5}, the strong maximum principle provides that $v_n \geq v_1$, and also $v_1$ is stricly positive on the compact subsets of $\Omega$, therefore $u_n \geq C_{\omega}>0$ a.e. in $\omega=\mathrm{supp}(\phi)$. Moreover, since $u_n$ converges to $u$ weakly in $X_0^{s_1,p}(\Omega)$, and $\Psi_{\eta}$ tends to zero in $X_0^s(\Omega)$, and also a.e. in $\Omega$, as $\eta \to 0^+$, passing to the limit, we have
\begin{equation}
\label{ReEq7}
\begin{aligned}
& \lim_{\eta \to 0^+} \lim_{n \to + \infty} \int_{\mathbb{R}^N} (-\Delta)^{\frac{s}{2}} u_n (-\Delta)^{\frac{s}{2}} [\phi (1- \Psi_{\eta})] \, dx =\int_{\mathbb{R}^N} u (-\Delta)^{s} \phi \, dx, \\
& \lim_{\eta \to 0^+} \lim_{n \to + \infty} \int_{\Omega} \frac{K(x) \phi (1- \Psi_{\eta})}{u_n^q}\, dx = \int_{\Omega} \frac{K(x) \phi}{u^q}\, dx,\\
& \lim_{\eta \to 0^+} \lim_{n \to + \infty} \int_{\Omega} \frac{f_n \phi (1-\Psi_{\eta})}{(u_n+\frac{1}{n})^{\gamma}} \, dx = \int_{\Omega} \frac{f \phi}{u^{\gamma}} \, dx, \\
& \lim_{\eta \to 0^+} \lim_{n \to + \infty} \int_{\Omega} \phi (1- \Psi_{\eta}) \mu_n \, dx= \int_{\Omega} \phi \mu \, dx.
\end{aligned}
\end{equation}
Besides, using \eqref{ReEq2} together with the same argument as in the proof of Theorem \ref{ReThm1}, gives
$$ 0 \leq \limsup_{n \to +\infty} \Bigg| \int_{\Omega} \frac{g_n \phi (1- \Psi_{\eta}) }{(u_n+\frac{1}{n})^{\gamma}} \, dx \Bigg| \leq \limsup_{n \to + \infty} \frac{\|\phi\|_{L^{\infty}(\Omega)}}{C_{\omega}^{\gamma}} \int_{\Omega} g_n (1-\Psi_{\eta}) \, dx \leq \frac{\|\phi\|_{L^{\infty}(\Omega)}}{C_{\omega}^{\gamma}} \eta. $$
This implies
\begin{equation}
\label{ReEq8}
\lim_{\eta \to 0^+} \lim_{n \to + \infty} \int_{\Omega} \frac{g_n \phi (1-\Psi_{\eta})}{(u_n+\frac{1}{n})^{\gamma}} \, dx =0. 
\end{equation}
Finally \eqref{ReEq6}, \eqref{ReEq7}, and \eqref{ReEq8} implies that
$$ \int_{\mathbb{R}^N} u (-\Delta)^{s} \phi \, dx = \int_{\Omega} \frac{K(x) \phi}{u^q}\, dx + \int_{\Omega} \frac{f \phi}{u^{\gamma}} \, dx + \int_{\Omega} \phi \mu \, dx, \qquad \forall \phi \in \mathcal{T}(\Omega).$$
This completes the proof of Theorem.
\end{proof}

\begin{remark}
Recalling Remark \ref{REM1} and also considering the above Relaxation Theorem we have proved the existence of solutions to problem \eqref{Eq1} for any datum of the form $f+T+\nu$, with $f \in L^1(\Omega)$, $T \in X^{-s}(\Omega)$, and $\nu$ a measure singular with respect to $(s,2)$-capacity, at least for some specific regular $\mu$. 
\end{remark}

Finally, we want to discuss the existence of entropy solution for the datum $ \mu \in L^1(\Omega) + X^{-s}(\Omega)$. Before, we need to recall the following definition from the literature.

\begin{definition}
We say that measure $\mu$ is diffuse with respect to $(s,2)$-capacity, if and only if $\mu(A)=0$, for every Borel set $A$ with zero $(s,2)$-capacity. 
\end{definition}

The following proposition characterizes the diffuse measures with respect to $(s,2)$-capacity. The proof can be obtained by adapting the proof of \cite[Theorem 2.1]{MR1409661}.
\begin{proposition}
\label{Proposition1-4}
Let $\mu$ be a bounded Radon measure in $\Omega$. Then the following statements are equivalent.
\begin{enumerate} 
\item
$\mu \in L^1(\Omega) + X^{-s}(\Omega)$.
\item
$\mu$ is diffuse with respect to $(s,2)$-capacity.
\end{enumerate}
\end{proposition}

Finally, by mimicking the proof of \cite[Proposition 2.1]{MR3659372}, we have the following proposition.
\begin{proposition}
\label{Proposition2-4}
If $\mu$ is diffuse with respect to $(s,2)$-capacity, then there exists a nondecreasing sequence $(\mu_n)_{n \in \mathbb{N}}$ of smooth bounded functions, with compact support in $\Omega$ such that
\begin{itemize}
\item
for every $n \in \mathbb{N}$, there exists $C_n>0$ such that for every $\xi \in C_{c}^{\infty}(\Omega)$,
$$ \Bigg| \int_{\Omega} \xi \, d\mu_n \Bigg| \leq C_n \|(-\Delta)^{\frac{s}{2}} \xi \|_{L^2(\Omega)}. $$
\item
$(\mu_n)_{n \in \mathbb{N}}$ converges strongly to $\mu$ in the strong topology of the space of finite Radon measures, i.e.
$$ \int_{\Omega} d |\mu_n| \to \int_{\Omega} d |\mu |. $$
\end{itemize}
\end{proposition}

\begin{corollary}
By noticing the Remark \ref{RemEntropy}, and taking advantage of Proposition \ref{Proposition1-4}, and Proposition \ref{Proposition2-4}, the existence of the entropy solution is guaranteed for the diffuse measures $\mu$.
\end{corollary}

\begin{remark}
It can be shown that every non-negative bounded Radon measure $\mu$ can be uniquely decomposed as $\mu=\lambda+\nu$, with $\lambda$ diffuse with respect to $(s,2)$-capacity, and $\nu$ singular with respect to the $(s,2)$-capacity (see for example \cite{MR1409661}). By Proposition \ref{Proposition1-4}, the measure $\lambda$ can be further decomposed as $\lambda=f+T$, with $f \in L^1(\Omega)$ and $T \in X^{-s}(\Omega)$. This decomposition is not unique since $L^1(\Omega) \cap X^{-s}(\Omega) \neq \{0\}$.
\end{remark}

\begin{remark}
In a similar way to the results in \cite{MR3639996, MR2989693, MR4160003}, and with the similar techniques, the fractional $p$-Laplacian counterpart of our results may be obtained for $p >1$, and $0 \leq \beta < ps$. Notice that what we really need is to have a monotone differential operator such that the maximum principle holds. See \cite[Theorem 1]{MR4191669} for a weak comparison principle for the weighted fractional $p$-Laplacian equations.
\end{remark}

\section*{References}

\bibliographystyle{spmpsci}
\bibliography{mybibfile}

\end{document}